\documentclass{article}

\usepackage{amsmath}
\usepackage{amssymb}
\usepackage{amsthm}
\usepackage{enumerate}
\usepackage{tikz-cd}
\usepackage{pgf,tikz}
\usetikzlibrary{arrows}
\usepackage{graphics}
\usepackage{commath}
\usepackage{epstopdf}
\usepackage{verbatim}
\usepackage{mathrsfs}
\usepackage{breqn}
\usepackage{authblk}
\usepackage{fancyhdr}

\usepackage{float}

\usepackage{xfrac}

\usepackage{geometry}
\usepackage[pagebackref=true,hypertexnames=false]{hyperref}

\newtheorem{TT}{Theorem}[section]
\newtheorem{LL}[TT]{Lemma}
\newtheorem{CC}[TT]{Corollary}
\newtheorem{PP}[TT]{Proposition}
\newtheorem{DD}[TT]{Definition}

\newtheorem{RR}[TT]{Remark}

\title{On the parametrised Whitehead torsion of families of nearby Lagrangian submanifolds}
\author{Sylvain Courte and Noah Porcelli}

\newcommand{\Addresses}{{
  \bigskip
  \footnotesize

  \textsc{}\par\nopagebreak
  \textit{} 
  \texttt{}

}}

\date{}

\usepackage{xcolor}
	\newcommand{\Sylvain}[1]%
	{\begin{quotation}{\footnotesize\textcolor{blue}{\textbf{Sylvain:} #1}}\end{quotation}}
	\newcommand{\Noah}[1]%
	{\begin{quotation}{\footnotesize\textcolor{red}{\textbf{Noah:} #1}}\end{quotation}}

\newcommand{\Wh}{\mathrm{Wh}}

\newcommand{\Dim}{\mathrm{Dim}}

\newcommand{\Id}{\mathrm{Id}}
\newcommand{\R}{\mathbb{R}}

\newcommand{\N}{\mathbb{N}}
\newcommand{\Z}{\mathbb{Z}}
\newcommand{\bP}{\mathbb{P}}
\newcommand{\G}{\mathrm{G}}

\newcommand{\crit}{\mathrm{Crit}}
\newcommand{\Diff}{\mathrm{Diff}}
\newcommand{\GL}{\mathrm{GL}}
\newcommand{\Homeo}{\mathrm{Homeo}}
\newcommand{\Top}{\mathrm{Top}}
\renewcommand{\int}{\mathrm{int}}

\def\bA{\mathbb{A}}

\def\bC{\mathbb{C}}

\def\bE{\mathbb{E}}
\def\bF{\mathbb{F}}

\def\bL{\mathbb{L}}
\def\bM{\mathbb{M}}
\def\bN{\mathbb{N}}

\def\bP{\mathbb{P}}
\def\bQ{\mathbb{Q}}
\def\bR{\mathbb{R}}

\def\bV{\mathbb{V}}
\def\bW{\mathbb{W}}
\def\bX{\mathbb{X}}

\def\bZ{\mathbb{Z}}

\newcommand{\cH}{\mathcal{H}}

\newcommand{\cL}{\mathcal{L}}

\newcommand{\cP}{\mathcal{P}}

\newcommand{\cS}{\mathcal{S}}

\newcommand{\eps}{\varepsilon}

\newcommand{\pt}{\mathrm{pt}}

\renewcommand{\del}{\partial}
\newsavebox{\pullback}
\sbox\pullback{%
\begin{tikzpicture}%
\draw (0,0) -- (1ex,0ex);%
\draw (1ex,0ex) -- (1ex,1ex);%
\end{tikzpicture}}

\numberwithin{equation}{section} 

\begin{document}
    \maketitle
    \pagenumbering{arabic}
    \begin{abstract}   
        Motivated by the strong nearby Lagrangian conjecture, we constrain the parametrised Whitehead torsion of a family of closed exact Lagrangian submanifolds in a cotangent bundle. We prove the parametrised Whitehead torsion admits a factorisation through simpler maps, in particular implying it is trivial on $\pi_0$, $\pi_1$, and that its image is divisible by the Euler characteristic. We provide concrete implications for the Lagrangian monodromy question in the case of a high-dimensional torus.

        This generalises earlier work of Abouzaid and Kragh \cite{AbKr} on the $\pi_0$ version, using different methods. Our main tool is the theory of twisted generating functions, building on \cite{ACGK}.
    \end{abstract}
    \tableofcontents
\section{Introduction}

   \subsection{Main results}\label{subsec:mainintro}
Let $M$ be a closed connected manifold. A \emph{nearby Lagrangian} in $T^*M$ is a closed exact Lagrangian submanifold in $T^*M$. In this paper, we study, $\cL(M)$ the space of closed connected exact Lagrangian submanifolds of $T^* M$ (see Section~\ref{sec:nearbylag} for a more precise definition). 

The \emph{nearby Lagrangian conjecture} asserts that the space $\cL(M)$ is path-connected, namely any such
Lagrangian submanifold is Hamiltonian isotopic to the zero-section, and in particular diffeomorphic to $M$. This paper is mainly concerned
with the \emph{strong} nearby Lagrangian conjecture, which asserts that the space $\cL(M)$ is in fact contractible.
No counter-examples are known even for this refinement. To our knowledge, the only cases where the strong
conjecture is proved are: $M=D^1$, $S^1$ (folkloric) and $D^2$ (Eliashberg-Polterovich \cite{ElPo}). In the cases here where $M$ has boundary, we assume additionally that the nearby Lagrangians coincide with the 0-section near the boundary.

A theorem of Abouzaid and Kragh \cite{Abouzaid:homotopy,Kragh} (see also Guillermou \cite{Guillermou:cotangent}) asserts that for $L\in \cL(M)$,
the projection $g_L : L \to M$ is a homotopy equivalence.

Another result of Abouzaid and Kragh \cite{AbKr} further says that $g_L$ is a \emph{simple} homotopy equivalence.
This amounts to say that the Whitehead torsion $\tau(g_L)$ of $g_L$-- a purely algebraic invariant living in the Whitehead group $\Wh(\pi_1 M)$-- vanishes. This is proved using Lagrangian Floer theory with $\bZ[\pi_1]$ coefficients. 

Our aim in this paper is to provide a generalisation of this vanishing result in the parametrised setting. The Whitehead group is isomorphic to the group of path components $\pi_0 \cH(M)$, where $\cH(M)$ is the space of \emph{stable $h$-cobordisms on $M$} (cf. Section \ref{sec: hcob spac}), and the correct parametrised analogue of the Whitehead torsion is a map $w: \cL(M) \to \cH(M)$, constructed in \cite{WW:I}, cf. also Section \ref{sec: para whit}. 
\begin{RR}
    By definition, $w$ factors through a map $w: \cS(M) \to \cH(M)$, where $\cS(M)$ is the \emph{structure space} considered in geometric topology. Roughly speaking, the structure space is the moduli space of pairs $(L, g)$, where $L$ is a closed manifold and $g: L \to M$ is a homotopy equivalence. 
    
    There is a canonical forgetful map $\cL(M) \to \cS(M)$ sending $L \subseteq T^*M$ to $(L, g_L)$. This map remembers all of the intrinsic topology of nearby Lagrangians-- including their smooth structures-- but throws away all symplectic information, such as the stable Gauss map of $L$.
\end{RR}

The simple homotopy result of \cite{AbKr} is equivalent to $w$ landing in the trivial path component of $\cH(M)$, i.e. $\pi_0 w = 0$. The strong nearby Lagrangian conjecture would in particular imply that the map $w: \cL(M) \to \cH(M)$ is nullhomotopic.

Using the twisted generating functions developed in \cite{ACGK} (see also \cite{AACK}) we prove the following theorem.

\begin{TT}\label{thm:main}
    Let $M$ be a closed connected manifold. Let $p:\cH(\pt)\to \cH(M)$ be the map which takes a stable $h$-cobordism
    on a point to its product by $M$. For any finite CW-complex $B$ and any map $\varphi:B\to \cL(M)$ there is a map $\delta:B\to \cH(\pt)$ 
    such that $w\circ \varphi$ is homotopic to $p\circ \delta$, namely we have
    a homotopy commutative diagram:
    \begin{equation}
    \begin{tikzcd}
    \cL(M) \arrow[r,"w"] &\cH(M) \\
    B \arrow[r,"\delta"]\arrow[u,"\varphi"] &\cH(\pt) \arrow[u,"p"] 
    \end{tikzcd}
    \end{equation}
\end{TT}

To our knowledge, this is the first general constraint on the topology of the trivial path component of the space $\cL(M)$.

\begin{RR}\label{rem:globaldelta}
It is likely that the map $\delta$ can be defined on the whole space as a map 
\[\delta:\cL(M)\to \cH(\pt)\]
such that $p\circ\delta$ is homotopic to $w$.
That would require showing that the construction of $\delta$ is sufficiently natural to be extended inductively
over a skeleton of $\cL(M)$.
\end{RR}

\begin{CC}\label{cor: main}
    \begin{enumerate}
        \item \cite{AbKr} The map $w$ is trivial on $\pi_0$.
        \item The map $w$ is trivial on $\pi_1$.
    \end{enumerate}

\end{CC}
In particular, we obtain a new proof of the main result of \cite{AbKr}.
\begin{proof}
It follows from Smale's $h$-cobordism theorem that $\pi_0 \cH(\pt)=0$ and from Cerf's pseudo-isotopy theorem
that $\pi_1 \cH(\pt)=0$, so we obtain both results from the factorisation through $\cH(\pt)$.
\end{proof}

In Section~\ref{sec:product} we prove a structural result about the map $p:\cH(\pt) \to \cH(M)$ which allows us to derive
more consequences of Theorem~\ref{thm:main}.
The stable $h$-cobordism space $\cH(M)$ admits a grouplike monoid structure (in fact an infinite loop space structure)
and it thus makes sense to consider the multiplication 
\[\cH(M)\xrightarrow{d} \cH(M)\] for any $d\in \Z$, see Subsection~\ref{subsec:monoid}. Moreover
there are functoriality maps $\cH(N)\to \cH(M)$ for any map of manifolds $N \to M$, see Subsection~\ref{subsec:functor}.
We denote $i:\cH(\pt)\to \cH(M)$ the map induced by the inclusion of a point in $M$.

\begin{TT}\label{thm:producteuler}
    Let $p: \cH(\pt) \to \cH(M)$ be as in Theorem \ref{thm:main}. Then $p$ is weakly homotopic to $i\circ\chi(M)$ where $\chi(M)$ is the Euler characteristic of $M$, namely we have a diagram:
    \begin{equation}
        \begin{tikzcd}
            &\cH(M) \\
            \cH(\pt) \arrow[ru,"p"] \arrow[r,"\chi(M)"] & \cH(\pt) \arrow[u,"i"]
        \end{tikzcd}
    \end{equation}
    which commutes up to weak homotopy: it homotopy commutes when restricted along any map $\delta: B \to \cH(\pt)$ from a finite CW complex $B$.
\end{TT}

In particular for manifolds $M$ with vanishing Euler characteristic (tori, odd dimensional spheres, etc)
we obtain the full parametric generalisation of Abouzaid-Kragh's simple homotopy equivalence theorem.
\begin{CC}\label{cor:maineulerzero}
Let $M$ be a closed connected manifold with $\chi(M)=0$. The map 
\[w:\cL(M)\to \cH(M)\]
is weakly nullhomotopic, namely for any map $\varphi :B \to \cL(M)$ from a finite CW-complex $B$, the map $w\circ\varphi$ is nullhomotopic.
\end{CC}

\begin{RR}\label{rem:CWorMFD}
In the statements of Theorem~\ref{thm:main}, Theorem~\ref{thm:producteuler} and Corollary~\ref{cor:maineulerzero}
we can euivalently replace the finite $CW$-complex $B$ by a compact smooth manifold $B'$, since any
such finite complex is homotopy equivalent to a compact smooth manifold and conversely.
\end{RR}
\begin{RR}
Weakly nullhomotopic does not in general imply nullhomotopic; counterexamples are so-called \emph{phantom maps} \cite[Section III]{Rudyak}.
However in view of Remark~\ref{rem:globaldelta}, we believe the stronger statement that $w$ is nullhomotopic also holds.
\end{RR}

\subsection{Applications to Lagrangian monodromy}

There is a tautological smooth fibre bundle over $\cL(M)$ whose fibre over a nearby Lagrangian $L \subseteq T^*M$ is exactly $L$. Letting $\cL_0(M)$ be the path component of the 0-section, this bundle is then classified by a map $\cL_0(M) \to B\Diff(M)$. It is tempting to use Theorem \ref{thm:main} to obtain constraints on this map.

For instance if we are given a loop $(L_t)_{t\in [0,1]}$ of exact Lagrangian submanifolds which is based at the zero-section, i.e. $L_0=L_1=M$,
there is a \emph{monodromy} (isotopy class of) diffeomorphism $\phi$ living in the mapping class group $\pi_0\Diff(M)$. In terms of the classifying map above, this is the image of the loop $(L_t)_{t\in [0,1]}$
under the morphism
\[\pi_1\cL_0(M)\to \pi_1 B\Diff(M)=\pi_0\Diff(M).\]
From the mere fact that the projection $L_t\to M$ is a homotopy equivalence for all $t$
we learn that $\phi$ must be \emph{homotopic} to the identity. We can then ask the subtler questions of whether $\phi$ is \emph{pseudo-isotopic} to the identity and then whether $\phi$ is \emph{isotopic} to the identity. 

Though we have not been able to prove such statements in the general case, from Corollary \ref{cor: main}(2) we are able to deduce restrictions on such monodromies in the case where $M$ is a torus. The mapping class group of high-dimensional tori was computed by Hatcher and Hsiang-Sharpe:

\begin{TT}[{\cite[Theorem 4.1]{Hat78}, \cite[Theorem 2.5]{Hsiang-Sharpe}}]\label{thm: hat}
    Assume $n \geq 6$\footnote{\cite[Theorem 4.1]{Hat78} states this result for $n=5$ too. There the $n=5$ case is attributed to announced work of Igusa, but to our knowledge has not appeared in the literature. }. Then there is an isomorphism of groups:
    
    \begin{equation}\label{eq: MCG Tn}
        \pi_0 \Diff(T^n) \cong (\GL_n(\Z) \ltimes (\Z/2)^\infty) \times F,
    \end{equation}
where $F$ is a finite group and with respect to which we have:
\begin{itemize}
\item $\pi_0 \Homeo(T^n)\cong \GL_n(\Z) \ltimes (\Z/2)^\infty$ and the projection forgetting $F$ is the forgetful map
$\pi_0\Diff(T^n)\to \pi_0 \Homeo(T^n)$. The finite subgroup $\{0\}\times F$ of $\pi_0\Diff(T^n)$ thus corresponds to diffeomorphisms that are isotopic to the identity map through \emph{homeomorphisms},
\item the further projection to $\GL_n(\Z)$ is the forgetful map $\pi_0\Diff(T^n)\to \pi_0 \G(T^n)$,
\item the normal subgroup $(\Z/2)^\infty\times \{0\}$ of $\pi_0\Diff(T^n)$ corresponds to all diffeomorphisms that are \emph{pseudo-isotopic} (cf. Definition \ref{DD:pseu}) to the identity map.
 
\end{itemize}
\end{TT}

We give a more detailed description of this decomposition in Section \ref{sec: mcg}. We remark that the size of the finite group $F$ in (\ref{eq: MCG Tn}) grows superexponentially in $n$.

\begin{CC}\label{cor: mcg tori}
    Let $L \cong T^n \subseteq T^* T^n$ be the zero section. Let $\phi \in \Diff(M)$ be the monodromy of some loop of exact Lagrangians in $T^* T^n$ based at $L$.

    Then under (\ref{eq: MCG Tn}), the isotopy class $[\phi] \in \pi_0\Diff(M)$ projects to 0 in $GL_n(\Z) \ltimes (\Z/2)^\infty$, or in other words,
$\phi$ is isotopic to the identity map through \emph{homeomorphisms}.
\end{CC}
This rules out all but finitely many potential monodromies. Previous results on this monodromy question were obtained by Ekholm, Kragh and Smith and by Dimitroglou-Rizell and Evans (see \cite[Theorem 1.4]{EKS} and \cite[Corollary 5.1]{ER14}) in the case of a sphere. Their approach is to convert the loop into a single Lagrangian embedding
in one higher dimension and to use previous results obstructing the existence of a Lagrangian exotic sphere in the cotangent bundle of a standard sphere. In some sense, these results are orthogonal to ours as they would be concerned with the projection to the subgroup $F$ of (\ref{eq: MCG Tn}) instead of the projection to $\GL_n(\Z) \ltimes (\Z/2)^\infty$. Following their strategy, some information about the projection to $F$ could potentially be obtained from the results of \cite{ACGK}.

    \begin{RR}[How much does $w$ detect in general?]
        Roughly speaking, the surgery exact sequence \cite{Wall} along with (\ref{eq: fib 2}) together decompose $\pi_*\cS(M)$ into three pieces: these are called (1). \emph{$L$-groups} $L_*(\bZ[\pi_1 M])$, (2). \emph{(higher) normal invariants} $\pi_* M^{G/O}$, and (3). the homotopy groups of the space $\pi_*\widetilde\Diff(M)/\Diff(M)$, cf. Definition \ref{def:spac}. 

        Weiss-Williams \cite[Section 1.5]{WW:Survey} show that the parametrised Whitehead torsion is actually an injection on the third of these groups $\pi_*\widetilde\Diff(M)/\Diff(M)$ in degrees $\lesssim \frac13 \Dim(M)$ when localised at odd primes. As we see in the example of tori, some interesting information in low degrees is captured at the prime 2 too.
    \end{RR}

\subsection*{Acknowledgements}

The first author is partially supported by ANR grant no. 21-CE40-0002 (COSY). The second author is supported by EPSRC grant EP/W015889/1.

The first author thanks Daniel \'Alvarez-Gavela and Thomas Kragh for stimulating discussions.
    The second author thanks Samuel Mu\~noz-Echaniz and Thomas Kragh for helpful discussion.

\section{Families of $h$-cobordisms}\label{sec: hcob}
    In this section, we recap some background and properties of spaces of $h$-cobordisms, primarily following \cite{WJR}. We work everywhere with \emph{smooth}, as opposed to topological, manifolds and $h$-cobordisms, though we do consider them later in Section \ref{sec: mcg}.
    \subsection{$h$-cobordisms}
        Let $M$ be a compact manifold.
        \begin{DD}
            
            An \emph{$h$-cobordism on $M$} is a pair $(W, \iota)$, where $W$ is a compact manifold with corners, and $\iota: M \to W$ is the inclusion of a smooth boundary face. Let $N = \overline{\partial W \setminus \operatorname{Im}(\iota)}$. We require that
            \begin{itemize}
                \item $\partial W = \operatorname{Im}(\iota) \cup N$ is a union of two codimension zero submanifolds along a shared boundary which constitutes the codimension 2 corner stratum of $W$.

                \item The inclusions $M, N \hookrightarrow W$ are both homotopy equivalences.
            \end{itemize}

            By abuse of notation, we will generally write $(W, M)$ or just $W$ when the choice of $\iota$ is implicit.
        \end{DD}
        
        Let $B$ be a compact manifold with corners and $\bM \to B$ be a smooth fibre bundle with fibres diffeomorphic to $M$. We will often denote this by $\bM = \{M_b\}_b$ for the collection of fibres, though note that we do still require the data of a smooth structure on the total space of the bundle.
        \begin{DD}
            A \emph{smooth family of $h$-cobordisms} on $\bM$ is a smooth fibre bundle $\bW = \{W_b\}_b$ over $B$, along with a smooth embedding $\bM \hookrightarrow \bW$ of fibre bundles over $B$ (again, we often write this as $\{M_b \hookrightarrow W_b\}_b$, again requiring that the map is smooth on total spaces), such that each fibre $(W_b, M_b)$ is an $h$-cobordism.

            If the family $\bM = M \times B$ is trivial, we call this a \emph{smooth family of $h$-cobordisms on $M$}.

            We write $H_B(\bM)$ or $H_B(M)$ for the set of isomorphism classes of smooth families of $h$-cobordisms on $M$ or $\bM$, over $B$.
        \end{DD}

 \begin{DD}
            Let $W$ be an $h$-cobordism from a compact manifold $M$ to $M'$. A \emph{classical right inverse} to $W$ consists of an $h$-cobordism $W^c$ on $M'$, along with a diffeomorphism $W \cup_{M'} W^c \cong M \times I$ relative to $M \times \partial I$.

            For a family of $h$-cobordisms $\bW$ from $\bM$ to $\bM'$, a classical right inverse is defined similarly as a family $\bW^c$ of $h$-cobordisms on $\bM'$ along with a family of diffeomorphisms $\bW \cup_{\bM'} \bW' \cong \bM \times I$ relative to $\bM \times \partial I$. The notion of left inverse is defined similarly.
        \end{DD}
        Note that even if the family $\bM \cong \{M\}_b$ is trivial, the other end $\bM'$ of the family $\bW$ might not be.
	\begin{RR}
        By a \emph{family of diffeomorphisms} between families of manifolds, we always mean a fibered diffeomorphism between the total spaces
        of the families.
	\end{RR}
        \begin{RR}
            In Section \ref{subsec:monoid} we encounter another notion of inverse of an $h$-cobordism; for $h$-cobordisms on $M \times I$, we prove in Lemma \ref{lem:inverses} that these notions agree. However, in our main arguments, both notions arise naturally.
        \end{RR}

	\begin{LL}\label{lem:classicalinverses}
            Any $h$-cobordism $W$ of dimension at least 6 admits both classical right and left inverses, which are isomorphic and unique up to isomorphism.

            The same holds for a family $\bW$ of such $h$-cobordisms on a fixed manifold $M$.
	\end{LL}
\begin{proof}
Let $(W_b,M,M'_b)$ be a family of $h$-cobordisms. We can assume $B$ is connected and pick a basepoint $a\in B$.
If $B$ is reduced to $\{a\}$ the result is standard: we consider an $h$-cobordism $W'$ from $M'$ to $M''$ with opposite Whitehead torsion
(under the identification $\pi_1 M \to \pi_1 M'$ induced by $W$) and use a gluing formula for Whitehead torsion
and the s-cobordism theorem to conclude that $W\cup_{M'} W'$ is diffeomorphic to $M\times [0,1]$ and that $M''$ is diffeomorphic to $M$.
Uniqueness up to isomorphism follows from the existence of a left inverse $W''$ and associativity of gluing (or rather composing)
cobordisms. It also follows that left and right inverse agree and are unique up to isomorphism. Indeed the same argument shows the existence of a left inverse $(W'',M',M)$ and we have the following diffeomorphisms rel boundary:
\[W''\simeq W'\cup_{M'} W\cup_M W' \simeq W'.\]

We now explain how to derive the family case. First we treat the case where each $W_b$ is trivial.
In this case the contractibility of the space of collars of $M\times \{0\}$ in $M\times [0,1]$ allows
to find a family of embeddings $i:W_b\to M\times [0,1)$ relative to $M$, and we just set $W^c_b=\overline{(M\times [0,1])\setminus W_b}$.
In the general case, we pick an inverse $h$-cobordism $W'_a$ for $W_a$ and consider the family
of trivial $h$-cobordisms $(W'_a\cup_M W_b)_b$. By the previous case, the latter family admits an inverse family of $h$-cobordisms $(W''_b)_b$ from $M'_b$ to $M'_a$. We then check that $(W''_b\cup_{M'_a} W'_a)_b$ is an inverse for $(W_b)_{b\in B}$
via the following diffeomorphisms rel boundary and fibered over $B$:
\[W_b\cup_{M'_b} W''_b\cup_{M'_a} W'_a\simeq W_a\cup_{M'_a} W'_a \cup_M W_b\cup_{M'_b} W''_b\cup_{M'_a} W'_a \simeq W_a \cup_{M'_a} W'_a \simeq M\times I.\]
The uniqueness statement in the family case is proven as in the case where $B$ is a single point.
 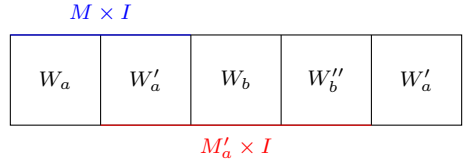
\begin{figure}[H]  
        \centering
        \begin{tikzpicture}[scale=.6,
        v/.style={draw,shape=circle, fill=black, minimum size=1.3mm, inner sep=0pt, outer sep=0pt},
        vred/.style={draw,shape=circle, fill=red, minimum size=1mm, inner sep=0pt, outer sep=0pt},
        vsmall/.style={draw,shape=circle, fill=black, minimum size=1mm, inner sep=0pt, outer sep=0pt}]

	\draw (0, 0) rectangle (10, 2);
	\draw (2, 0) to (2, 2);
	\draw[blue] (0,2) to (4,2);
	\draw (4, 0) to (4, 2);
	\draw (6, 0) to (6, 2);
	\draw (8, 0) to (8, 2);
	\draw[red](2,0) to (8,0);

	\node at (1, 1) {\footnotesize $W_a$};
	\node at (3, 1) {\footnotesize $W'_a$};
	\node[blue] at (2, 2.5) {\footnotesize $M\times I$};
	\node at (5, 1) {\footnotesize $W_b$};
	\node[red] at (5,-0.5) {\footnotesize $M'_a\times I$};
	\node at (7, 1) {\footnotesize $W''_b$};
	\node at (9, 1) {\footnotesize $W'_a$};

        \end{tikzpicture}
        \caption{Existence and uniqueness of classical inverses.}
        \label{fig: classical}
    \end{figure}
\end{proof}

\begin{DD}
Given an $h$-cobordism $W$ on $M$, we find, modulo smoothing the corners, another one $(W\times I, M\times I)$.
This defines for any compact manifold $B$ a map of sets $H_B(M)\to H_B(M\times I)$
and we set $\cH_B(M)=\operatorname{colim}_k H_B(M\times I^k)$.
\end{DD}

\subsection{$h$-cobordism spaces}\label{sec: hcob spac}

        \begin{DD}
            The \emph{$h$-cobordism space} of $M$ is the (geometric realisation of a) simplicial set $H(M)$, whose $k$-simplices $H(M)_k$ are defined to be the set of smooth families of $h$-cobordisms on $M$ parametrised by the standard $k$-simplex $\Delta^k$.
        \end{DD}
        \begin{RR}
            Strictly speaking, these are not sets. To resolve this, one must instead work with $h$-cobordisms which are embedded submanifolds in $\bR^\infty$; see \cite[Remark 1.1.2]{WJR} or \cite[Section 2.1]{HTW} for more details.
        \end{RR}
        \begin{RR}\label{rem:partitiondef}
            This differs from the definition considered in \cite{Wa82}, where $h$-cobordisms are required to be embedded in $M \times [0,1]$, or equivalently to be given a classical right inverse. However the corresponding $h$-cobordism spaces are homotopy equivalent when $\operatorname{Dim}(M) \geq 5$; as can be seen from the arguments in the proof of Lemma~\ref{lem:classicalinverses}.
	\end{RR}
        \begin{PP}\label{prop: bij iso rep}
            Let $B$ be a compact manifold. Then there is a canonical bijection of sets
            \begin{equation*}
                H_B(M) \cong [B, H(M)]
            \end{equation*}
            
            between isomorphism classes of smooth families of $h$-cobordisms on $M$ with homotopy classes of map $B \to H(M)$.
        \end{PP}
        \begin{proof}
            Follows from results in \cite[Section 2]{GRW:user's-guide}.
        \end{proof}
        The assignment $W \mapsto W\times I$ gives at the space level \emph{stabilisation maps} $\sigma: H(M) \to H(M \times I)$;
we refer to \cite[Section 1.1]{WJR} (or \cite{Wa82}, in a different model for the $h$-cobordism space) for more details. 
        \begin{DD}
            The \emph{stable $h$-cobordism space} $\cH(M)$ of $M$ is defined to be the colimit under repeated stabilisation:
            \begin{equation}\label{eq:stableH}
                \cH(M) := \operatorname{colim}_k H(M \times I^k)
            \end{equation}
        \end{DD}
        \begin{RR}\label{rmk: two stab}
            There are two different definitions of stabilisation maps $H(M) \to H(M \times I)$ in the literature, known in \cite{Wa82} as \emph{upper} and \emph{lower} stabilisation $\sigma^{+}, \sigma^{-}$ respectively, and either can be used in the definition of the stable $h$-cobordism space. 

            The two stabilisation maps $\sigma^{+}$ and $\sigma^{-}$ are not homotopic in general. However, they are related by a functorial autoequivalence of $H(M \times I)$, so the resulting stable $h$-cobordism spaces $\cH^{+}(M)$ and $\cH^{-}(M)$ are indeed homotopy equivalent.

            We follow the conventions of \cite{WJR} and \cite{Hat78}: we take $\sigma$ to be the lower stabilisation map $\sigma^{-}$. \cite{WW:Survey} uses upper stabilisation $\sigma^{+}$ everywhere, whereas \cite{Wa82} defines the stable $h$-cobordism space using a sequence of stabilisation maps which alternates between $\sigma^+$ and $\sigma^-$.
        \end{RR}
       
    \subsection{Functoriality}\label{subsec:functor}
        In this section, we recall how $\cH(\cdot)$ can be made into a functor from the category of compact manifolds to spaces.
In general, it can be extended to a functor from the category of spaces \cite{Wa82}, but this weaker form of functoriality is both simpler and sufficient for our purposes.

        Let $i: M \to N$ be a codimension 0 embedding of compact manifolds of dimension $n$.
        \begin{LL}\label{lem: hcob glue hcob}
            Let $W$ be an $h$-cobordism on $M$. Define $V$ to be the pair:
            \begin{equation}
                V = \left((N \times I) \cup_{M \times \{1\}} W, N \times \{0\}\right)
            \end{equation}
            where we equip this with a smooth structure as in Lemma \ref{lem: smooth structure on gluing}, and smooth corners implicitly.
            
            Then $V$ is an $h$-cobordism on $N$.
        \end{LL}
        \begin{proof}
            It is clear that the inclusion $N \to V$ is a homotopy equivalence. Let $M' = W \setminus \int(M)$, and $N' = (N \setminus \int(M)) \cup_{\partial M} M'$. Then $W$ is an $h$-cobordism from $M$ to $M'$, and $V$ is a cobordism from $N$ to $N'$.

            Since $M' \to W$ is an equivalence (relative to $\partial M$), so is the inclusion $N' \to (N \setminus \int(M)) \cup_{\partial M} W$. Then the inclusion of this second expression into $V$ is also an equivalence (the complement is $N \times [0,1)$).
        \end{proof}

        \begin{DD}
            We define a map $H(i): H(M) \to H(N)$ as follows. For a $k$-simplex $\bW = \{W_b\}_{b \in \Delta^k} \in H(M)_k$, we define $H(i)(\bW)$ to be the smooth family of $h$-cobordisms on $N$ over $\Delta^k$ defined by
            \begin{equation}\label{eq: glued hcob}
                \left\{
                \left(
                (N \times I) \cup_{M \times \{1\}} W_b, N \times \{0\}
                \right)
                \right\}_{b \in \Delta^k}
            \end{equation}
            Note these are $h$-cobordisms by Lemma \ref{lem: hcob glue hcob}.
        \end{DD}

        This can be made compatible with stabilisation and so induces a map on stable $h$-cobordism spaces $\cH(i): \cH(M) \to \cH(N)$.

        By Proposition \ref{prop: bij iso rep}, $H(i)$ induces a map $H_B(i): H_B(M) \to H_B(N)$ for any compact manifold $B$. We may describe this explicitly: this sends $\bW = \{W_b\}_{b \in B}$ to
        \begin{equation*}
            \left\{
            \left(
            (N \times I) \cup_{M \times \{1\}} W_b, N \times \{0\}
            \right)
            \right\}_b
        \end{equation*}
        again using Lemma \ref{lem: smooth structure on gluing} and smoothing corners implicitly.

For a general map $i:M\to N$, we approximate it by an embedding $M \to N\times D^k$ and take a tubular neighborhood $D_M \to N\times D^k$
which is now a codimension zero embedding, so we find a map $H(D_M)\to H(N\times D^k)$ which stably gives $\cH(M)\to \cH(N)$ (see \cite{Wa82} for more details). We use
this only for the inclusion of a point in $M$ (as in Theorem~\ref{thm:producteuler}) where it amounts to thickening the point to a disk in $M$
.

        We use two connectivity estimates (though these are significantly more refined than is required for our purposes):

        \begin{TT}[Igusa stability \cite{Igusa}]
            The stabilisation $\sigma: H(M) \to H(M \times I)$ is $k$-connected if $n \geq \operatorname{max}(2k+8, 3k+5)$.
        \end{TT}
        \begin{TT}[{\cite[Proposition 10.9]{BokMad}, \cite[1.1]{Waldhausen:TopI}}]
            The induced map on stable $h$-cobordism spaces $\cH(i): \cH(M) \to \cH(N)$ is $k$-connected if $i: M \to N$ is $(k+1)$-connected.
        \end{TT}
        \begin{CC}\label{cor: conn est}
            Let $B$ be a compact $d$-dimensional manifold. Assume $i: M \to N$ is $(d+2)$-connected and $\operatorname{Dim}(M)=\operatorname{Dim}(N) \geq \operatorname{max}(2d+12, 3d+11)$.
            
            Then $H_B(i): H_B(M) \to H_B(N)$ is a bijection.
        \end{CC}
        \begin{proof}
            The condition on $i$ implies $\cH(i)$ is $(d+1)$-connected, and the condition on $\operatorname{Dim}(M)$ implies the maps $H(M) \to \cH(M)$ and $H(N) \to \cH(N)$ are $(d+2)$-connected. Therefore $H(i)$ is $(d+1)$-connected, and we conclude by Proposition \ref{prop: bij iso rep}.
        \end{proof}

\subsection{The parametrised Whitehead torsion}\label{sec: para whit}

\begin{DD}\label{def:structurespaceB}
Let $B$ be a compact manifold, $M$ a closed manifold, $(L_b)_{b\in B}$ a family of
closed manifolds parametrised by $B$ and $g_b:L_b\to M$ a family of homotopy
equivalences. Two such families $(L_b,g_b)$ and $(L'_b,g'_b)$ are equivalent
if there exists a family of diffeomorphisms $\psi_b:L_b\to L'_b$ and a family of homotopies
between $g'_b\circ \psi_b$ and $g_b$.
We denote $\cS_B(M)$ the set of equivalence classes of such families.
\end{DD}

\begin{DD}\label{def:structurespace}
We define $\cS(M)$ to be (the geometric realisation of) the simplicial set whose $k$-simplices
are families $(L_b,g_b)_{b\in \Delta^k}$ as in Definition~\ref{def:structurespaceB} (not considered up to equivalence) and the face and degeneracy
maps are given by pullbacks along the standard face/degeneracy maps between simplices. This is called the \emph{structure space} of $M$.
\end{DD}

\begin{PP}\label{prop:classifyingS}
For any compact manifold $B$ and closed manifold $M$, there is a natural bijection
\[\cS_B(M) \to [B,\cS(M)].\]
\end{PP}

We will associate to an element of $\cS_B(M)$ a stable $h$-cobordism on $M$, well defined
up to isomorphism, namely an element of $\cH_B(M)$. When $B$ is a point, $\cH_B(M)=\Wh(\pi_1 M)$
and we recover the Whitehead torsion of a homotopy equivalence.

We will use the following standard lemma, and reproduce its proof for completeness.
\begin{LL}\label{lem:codim3hcob}
    Let $X$ be a closed connected manifold, and $X \to Y$ an embedding of codimension at least 3 into a compact manifold $Y$ which is a homotopy equivalence. Assume the inclusion $\del Y \to Y$ is an isomorphism on $\pi_1$.
    
    Then the complement $W=Y\setminus \int D_X$ of a tubular neighborhood $D_X$ of $X$ in $Y$ is an $h$-cobordism on $\del D_X$.
\end{LL}
\begin{proof}
    Since $k\geq 3$ the maps $\pi_1 \del D_X \to \pi_1 D_X$ and $\pi_1 W \to \pi_1 Y$ are isomorphisms (any map from a surface to $Y$ is homotopic to a map disjoint from $X$). Moreover since $\pi_1\del Y \to \pi_1 Y$ is an isomorphism the same holds for $\pi_1\del Y\to \pi_1 W$. We have $H_*(Y,D_X)=H_*(Y,X)=0$ with any local system and by excision the same holds for $H_*(W,\del D_X)$. By Hurewicz's theorem and Whitehead's theorem we then conclude that $\del D_X \to W$ is a homotopy equivalence. Using Poincaré-Lefschetz duality we obtain the same conclusion for the inclusion of the other boundary component $\del Y \to W$ (note that our assumptions imply that $\del Y$ is non-empty and connected).
\end{proof}

Pick $k\geq 3$ (large enough) and a family of embeddings $L_b\to \int (M \times D^k)$. Pick
a family of tubular neighborhoods $E_b$ of $L_b$ in $\int (M \times D^k)$. We define $W_b=(M\times D^k) \setminus \int E_b$. 
By Lemma~\ref{lem:codim3hcob} this is a family of $h$-cobordisms on $\del E_b$ or, seen from the other side, on $M\times S^{k-1}$.

For large enough $k$, the family of embeddings $L_b\to M\times D^k$ are isotopic, and thus the resulting $h$-cobordisms
$(W_b,M\times S^{k-1})$ are isomorphic. However to find a meaningful invariant we need to suppress the
dependence on $k$ so we require an extra step.

Pick a classical right inverse $(W'_b;M\times S^{k-1},\del E_b)$ to $W_b$ (see Lemma~\ref{lem:classicalinverses}), namely 
\[W_b\cup_{M \times S^{k-1}} W'_b\simeq \del E_b\times I.\]
We consider $E'_b=(M\times D^k)\cup_{M\times S^{k-1}} W'_b\simeq E_b$ and the pair 
\[(V_b,M\times D^k)=(E'_b\times I, M\times D^k\times \{0\}).\]

\begin{LL}
For each $b$, the pair $(V_b, M\times D^k)$ is an $h$-cobordism.
\end{LL}
\begin{proof}
We supress $b$ from the notations. $V$ deformation retracts onto $E'$ and then further to $M\times D^k$ since $W'$ is
an $h$-cobordism. For the other boundary component $\del_+ V=(E'\times \{1\}) \cup (\del E'\times I) \cup W'$, we observe the factorisation
\[E'\times \{1\}\cup ((\del E')\times I) \to \del_+ V \to E'\times I=V\]
where the map on the left and the composition are homotopy equivalences, hence so is $\del_+V \to V$.
\end{proof}

\begin{LL}\label{lem: stab comp w}
Under the natural embedding $D^k \to D^{k+1}$, the above construction of $(V_b\times I)$ gets stabilised
as in \eqref{eq:stableH}.
\end{LL}
\begin{proof}
We use the notations of the previous paragraph, but suppress the subscript $b$ for simplicity.
Pick a subinterval $J'=[1/3,2/3]\subset J=[0,1]$. We give subcripts $k$ to the object constructed
as above with the integer $k$. We have $E_k \subset M\times D^k$ and we thicken this to $E_{k+1}=E_k\times J' \subset M\times D^k\times J$.
Up to smoothing corners, $M\times D^k\times J\simeq M\times D^{k+1}$ and the $h$-cobordism $W_k=M\times D^k \setminus E_k$
becomes
\[W_{k+1}=(M\times D^k\times J) \setminus E_{k+1}.\]
To complete the construction at the range $k+1$ we need to pick a right inverse to $W_{k+1}$
but we observe that $W'_k\times J$ is such an inverse, where $W'_k$ is a right inverse of $W_k$.
Indeed, using that $W_k\cup W'_k$ is a trivial cobordism, we see that the embedding 
\[E_{k+1}=E_k\times J' \subset (M\times D^k\times J)\cup_{M\times S^{k-1}\times J} (W'\times J)=E'_k\times J = E'_{k+1}\]
is isotopic to a diffeomorphism.
Finally the $h$-cobordism $(V_k=E'_k\times I, M\times D^k\times \{0\})$ becomes 
\[V_{k+1}=(E'_{k+1}\times I, M\times D^k\times J'\times \{0\})=(E'_k\times I \times J, M\times D^k\times J'\times \{0\})\simeq V_k \times J,\]
which is precisely the stabilisation of $V_k$ as in Section~\ref{sec: hcob}.
\end{proof}

\begin{DD}\label{def:paramWh}
    The \emph{parametrised Whitehead torsion} of a family of homotopy equivalences
    $(L_b, g_b : L_b \to M)$ is the class of $(V_b,M \times D^k) \in \cH_B(M)$, constructed as above. This
defines a map
\[w:\cS_B(M)\to \cH_B(M).\]
which is well-defined by Lemma \ref{lem: stab comp w}.
\end{DD}

\begin{RR}\label{rem:spacelevelw}
Generalizing the arguments above we can construct $w$ at the space-level and obtain a map
\[\cS(M) \to \cH(M)\]
which gives the previous definition under the bijections from Proposition~\ref{prop:classifyingS} and Proposition~\ref{prop: bij iso rep}.
\end{RR}

\begin{RR}
    The definition of $w$ also depends on the convention for stabilisation as in Remark \ref{rmk: two stab}; since our convention for stabilisation differs from that of \cite{WW:Survey}, our definition of theirs also does. Our definition of parametrised Whitehead torsion is identical to that of \cite{WW:I}, \emph{except} for this difference in convention. In particular, under the autoequivalences that relates the two possible definitions of $\cH(M)$, the two definitions of the parametrised Whitehead torsion agree.
\end{RR}

    \subsection{Decompositions of total space}
        Let $N$ be a closed manifold.
        \begin{PP}\label{prop: tot hcob triv}
            Let $\pi: W \to C$ be a smooth fibre bundle over a compact connected manifold $C$, and $\iota: N \times C \to W$ an embedding of fibre bundles, such that for each $c \in C$ the fibre $(\pi^{-1}\{c\}, N)$ is an $h$-cobordism.  In particular, the total space $(W, N \times C)$ forms an $h$-cobordism.

            Assume further that each fibre $(\pi^{-1}\{c\}, N)$ is a trivialisable $h$-cobordism: it admits a diffeomorphism to $(N \times I, N \times \{0\})$ (crucially, we do \emph{not} assume there exists a family of such diffeomorphisms, varying smoothly in $c$).

            Then there is a diffeomorphism $(W, N \times C) \cong (N \times C \times I, N \times C \times \{0\})$ relative to $N \times C$, i.e. the total space of the bundle is a trivial $h$-cobordism.
        \end{PP}
        In general, the diffeomorphism produced will not respect the projection maps to $C$.
        \begin{proof}
            Let $\partial_- W = \operatorname{Im}(\iota)$ and let $\partial_+ W = \partial W \setminus \partial_- W$.
        
            We choose a Morse function $f:C\to \R$ together with a gradient-like vector field $Y$; we may choose these so that $Y$ points strictly inwards along $\partial C$.
            
            Above a neighborhood $U$ of the set $\crit(f)$ of critical points of $f$, we pick a trivialisation
            of the bundle $\varphi : U\times N \times [0,1] \overset{\sim}{\longrightarrow} \pi^{-1}(U)$.
            We pick a function $g:W\to [0,1]$ to be any function which has regular level sets $g^{-1}(0)=\del_- W$ and $g^{-1}(1)=\del_+ W$ and coincides with the projection to $[0,1]$
            in the trivialisation $\varphi$ above a neighborhood of $\crit(f)$. We then choose a fiberwise gradient-like
            vector field $X$ for $g$. Finally, we pick a lift $Y'$ of $Y$ which is horizontal above a neighborhood of $\crit(f)$ with respect to the trivialisation $\varphi$ and tangent to $\del W$.

            We claim that the vector-field $Z=X+Y'$ is gradient-like for the function $h^a=g+a\pi^*f$ for $a>0$ sufficiently large.
            Indeed we have:
            \[dh^a(Z)=dg(X)+dg(Y')+a\pi^*(df(Y)).\]
            We have $dg(X)\geq 0$ and $\pi^*df(Y)\geq 0$ everywhere. Near $\pi^{-1}(\crit(f))$, $dg(Y')=0$, $dg(X)>0$ so $dh^a(Z)>0$. Away from a neighborhood of
$\pi^{-1}(\crit(f))$, $\pi^* df(Y)>0$ and $dg(Y')$ is bounded, so $dh^a(Z)>0$ for sufficiently large $a>0$. Moreover $Z$ is transverse to $\del W$, inward pointing along $\del_- W$ and outward pointing along $\del_+ W$ since this is the case for $X$ and $Y'$ is tangent to $\del W$.
            
            So we conclude that all trajectories of $Z$ starting on $\del_- W$ go all the way to $\del_+ W$. After renormalizing $Z$, we then define a diffeomorphism $\del_-W \times [0,1] \to W$ uniquely by mapping $\del_-W \times \{0\}$ to $\del_- W$ by the identity map and mapping $\del_t$ (the vector field corresponding
            to the $t$-coordinate on $[0,1]$) to $Z$.
        \end{proof}
        
        A similar result holds when each fibre is not necessarily trivial:
        \begin{PP}\label{prop: tot hcob}

            Let $\pi: W \to C$ be a smooth fibre bundle over a compact connected manifold $C$, and $\iota: N \times C \to W$ an embedding of fibre bundles, such that for each $c \in C$ the fibre $(\pi^{-1}\{c\}, N)$ is an $h$-cobordism.

            Assume further that each $h$-cobordism $(\pi^{-1}\{c\}, N)$ is diffeomorphic to an $h$-cobordism $(V, N)$, and that $\operatorname{Dim}(V) \geq 6$.

            Then there is a diffeomorphism $(W, N \times C) \cong (V \times C, N \times C)$, i.e. the total space of the bundle is a product $h$-cobordism.
        \end{PP}
        More succinctly, this says that the total space of a bundle of $h$-cobordisms on a closed manifold $N$ splits as a product.
        \begin{proof}

            Let $N' = \partial V \setminus N$ be the other boundary component of $V$. Choose an $h$-cobordism $V'$ from $N'$ to $N$, along with diffeomorphisms $\phi: V' \cup_{N} V \cong N' \times I$ and $\phi': V \cup_{N'} V' \cong N \times I$ relative to the boundaries; this exists since $\operatorname{Dim}(V) \geq 6$ (this corresponds to a choice of inverse $h$-cobordism in the classical sense).

            Let $Z = (V' \times C) \cup_{N \times C} W$, and let $\pi': Z \to C$ be given by projection on the first term and $\pi$ on the second. $Z$ is now a smooth fibre bundle over $C$ with fibres which are trivial $h$-cobordisms on $N'$, so by Proposition \ref{prop: tot hcob triv}, there is a diffeomorphism $(Z, N' \times C) \cong (N' \times C \times I, N \times C \times \{0\})$. Gluing back on $V \times C$ along $N' \times C$ and using $\phi'$ then implies the result.
        \end{proof}
        The same result holds in families:
        \begin{PP}\label{prop: tot hcob param}
            Let $B$ and $C$ be compact connected manifolds, $\{\pi_b: W_b \to C\}_b$ a smooth family over $B$ of smooth fibre bundles over $C$, and $\iota: N \times C \times B \to \bW$ an embedding of fibre bundles. Assume for each $b \in B, c \in C$, the fibre $(W_{b,c}, N)$ is an $h$-cobordism. 

            Fix some $c_0 \in C$, and let $V_b = \pi_{b}^{-1}\{c_0\}$; these together assemble to form a smooth family $\bV=\{V_b\}_{b \in B}$ of $h$-cobordisms on $N$.

            Then there is an isomorphism of families of $h$-cobordisms on $C \times N$ over $B$:
            \begin{equation*}
                \left\{(W_b, N \times C) \cong (V_b \times C, N \times C)\right\}_{b \in B}
            \end{equation*}
        \end{PP}
        \begin{proof}

            All the steps in the proof of Proposition \ref{prop: tot hcob} (as well as that of Proposition \ref{prop: tot hcob triv}) can be carried out smoothly in families.
        \end{proof}

    \subsection{Monoid structure}\label{subsec:monoid}

        Once again, let $M$ and $B$ be compact manifolds. Assume $\operatorname{Dim}(M) \geq 5$. In this section, we equip $H_B(M \times I)$ with the structure of a group.

        \begin{DD}
            Let $W$ and $V$ be $h$-cobordisms on $M \times I$. We define $W+V$ to be the $h$-cobordism obtained by gluing $W$ and $V$ to $M \times I \times [0,1]$ along $M \times [0,1/2] \times \{1\} \cong M \times I \subseteq \partial W$ and $M \times [1/2,1] \times \{1\} \cong M \times I \subseteq V$ respectively.

            Performing the same construction in families defines a family of $h$-cobordisms $\{W_b+V_b\}_b$ associated to a pair of families $\{W_b\}_b$ and $\{V_b\}_b$ on $M \times I$.
        \end{DD}
        \begin{LL}
            The isomorphism class of $W+V$ only depends on the isomorphism classes of $W$ and $V$. In particular, $+$ descends to a well-defined map $+: H_B(M \times I) \times H_B(M \times I) \to H_B(M \times I)$.
        \end{LL}
        \begin{LL}
            The trivial $h$-cobordism $\{M \times I \times I\}_b$ is a unit for this operation. 
        \end{LL}

        \begin{LL}
            The operation $+$ on $H_B(M \times I)$ is associative.
        \end{LL}
        \begin{proof}
            Let $W,V,U$ be $h$-cobordisms on $M \times I$. $(W+V)+U$ and $W+(V+U)$ are both obtained by gluing on $W$, $V$ and $U$ to $M \times I \times I$ along a triple of subintervals in $I$ in the same order; by Lemma \ref{lem: isotopy diffeo}, it follows that they are isomorphic $h$-cobordisms on $M \times I$.
        \end{proof}

        \begin{LL}\label{lem:inverses}
            The monoid $(H_B(M \times I),+)$ admits inverses, i.e. is a group.
        \end{LL}
        \begin{proof}
We prove this when $B$ is a point; the general case is identical.

Let $(W;M\times I, N')$ be an $h$-cobordism on $M\times I$. Pick a classical right inverse $h$-cobordism $(W^c;N',M\times I)$.

Pick also a collar neighborhood $M\times I \subset \del W^c\setminus \int(N')$ of $M\times \{1\}$ (see Figure~\ref{fig: inverseU} where this collar is in red).
The manifold $N'$ is itself an $h$-cobordism from $M$ to $M$ and as such it admits an inverse $h$-cobordism $(N;M,M)$ (again in the classical sense). Decompose
$M\times [0,2]$ as $(M\times [0,1])\cup (M\times [1,2])$ and identify $M\times [1,2]$ with the concatenation $N\cup_M N'$.
Consider $N\times I$ attached to $N\subset M\times [1,2]$ along $N\times \{0\}$ and to $\del W^c$ along the collar neighborhood $M\times I$.
We define $U$ to be the union $W^c\cup_{M\times I} N\times I$ along the collar neighborhood $M\times I$ (see Figure~\ref{fig: inverseU} where $\del U$ is colored in blue). This manifold $U$ is an $h$-cobordism from $N'\cup N$ to $(M\times I)_{M\times \{1\}}\cup N$. The attaching region $N'\cup N$ of $U$ is diffeomorphic to $M\times I$
and it can therefore be isotoped to lie entirely inside $M\times [1,2]$, therefore by Lemma \ref{lem: isotopy diffeo} exhibiting the trivial cobordism $(M\times I \times [0,2]) \cup W\cup W^c \cup N\times I$ as
the sum $W + U$ (see the bottom of Figure~\ref{fig: inverseU}). Hence $U$ is an inverse to $W$ for the monoid structure of $H_B(M\times I)$.
   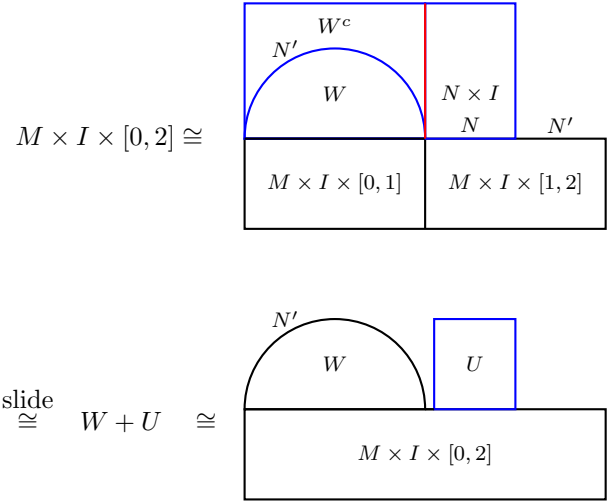
\begin{figure}[H]  
        \centering
        \begin{tikzpicture}[scale=.6,
        v/.style={draw,shape=circle, fill=black, minimum size=1.3mm, inner sep=0pt, outer sep=0pt},
        vred/.style={draw,shape=circle, fill=red, minimum size=1mm, inner sep=0pt, outer sep=0pt},
        vsmall/.style={draw,shape=circle, fill=black, minimum size=1mm, inner sep=0pt, outer sep=0pt}]

	\node at (-3, 0) {$M \times I \times [0,2] \cong$};

	\draw[thick] (0, 0) rectangle (8, -2);
	\draw[thick] (4, 0) to (4, -2);
	\draw[thick,blue] (4,0) rectangle (6,3);
	\draw[thick,blue] (0,0) rectangle (4,3);
	\draw[thick,blue] (0,0) -- (4,0) arc(0:180:2) --cycle;
	
	\draw[thick,black] (0,0) to (4,0);
	\draw[thick,red] (4,0) to (4,3);

	\node at (2, 1) {\footnotesize $W$};
	\node at (2, 2.5) {\footnotesize $W^c$};	
	\node at (5, 1) {\footnotesize $N\times I$};
	\node at (5,0.3) {\footnotesize $N$};
	\node at (7,0.3) {\footnotesize $N'$};
	\node at (0.9, 2) {\footnotesize $N'$};
	\node at (2, -1) {\footnotesize $M\times I \times [0,1]$};
	\node at (6, -1) {\footnotesize $M\times I \times [1,2]$};

	\node at (-3, -6) {$\overset{\text{slide}}{\cong} \quad W+U\quad\cong$} ;

	\draw[thick] (0, -6) rectangle (8, -8);
	\draw[thick,black] (0,-6) -- (4,-6) arc(0:180:2) --cycle;
	\draw[thick,blue] (4.2,-6) rectangle (6,-4);
	\draw[thick,black] (4,-6) to (4.2,-6);
	\node at (0.9, -4) {\footnotesize $N'$};
	\node at (5.1, -5) {\footnotesize $U$};
	\node at (4,-7) {\footnotesize $M\times I \times [0,2]$};
	\node at (2, -5) {\footnotesize $W$};

        \end{tikzpicture}
        \caption{The inverse $h$-cobordism $U\simeq -W$.}
        \label{fig: inverseU}
    \end{figure}
        \end{proof}

        \begin{RR}
            This construction can be enhanced to equip $H(M \times I)$ with the structure of an $\bE_1$-monoid, which is grouplike if $\operatorname{Dim}(M) \geq 5$. We do not require this for our purposes. 
        \end{RR}

\section{Product $h$-cobordisms}\label{sec:product}
    In this section, we prove a structural result about product $h$-cobordisms: 
    \begin{TT}\label{thm: prod decomp}
        Let $P$ be a compact $n$-manifold, and let $B$, $P$ and $M$ be compact manifolds. Let $\{W_b\}_{b \in B}$ be a smooth family of $h$-cobordisms on $M$. Let $\left\{R_b = (M \times P) \cup_{M \times \partial P} (W_b \times \partial P)\right\}_{b \in B}$. 
        
        Then there is an isomorphism of families of $h$-cobordisms on $\{R_b\}_b$:
        
        \begin{equation}
            \{W_b \times P\}_b \cong \left\{(R_b \times I) \cup \left(\chi(P,\partial P) \cdot (W \times D^n)\right)\right\}_b
        \end{equation}
        On the right-hand side, we glue in copies of $W \times D^n$ along $M \times D^n$ using $\chi(P, \partial P)$ disjoint embeddings $D^n \to P$. If $\chi(P, \partial P)$ is negative, we take this to mean we glue in $-\chi(P, \partial P$) copies of $-(W \times D^n)$.
    \end{TT}

\begin{RR}
Note that $-(W\times D^n)$ makes sense in view of Lemma~\ref{lem:inverses} as soon as $n\geq 1$. For $n=0$, $\chi(P,\del P)=\chi(P)\geq 0$, so we do not have
to consider the possibly undefined $-W$.
\end{RR}

    \begin{proof}[Proof of Theorem \ref{thm:producteuler} from Theorem \ref{thm: prod decomp}]
        For $P$ closed, Proposition \ref{prop: bij iso rep} and Theorem \ref{thm: prod decomp} together imply that the two maps 
        \begin{equation}
            \cdot \times P \textrm{ and } \chi(P,\partial P) \cdot H(\iota): \cH(M) \to \cH(M \times P)
        \end{equation}
        are weakly homotopic (see also Remark~\ref{rem:CWorMFD}), so we obtain Theorem~\ref{thm:producteuler}.
    \end{proof}
    \begin{RR}
        Theorem \ref{thm: prod decomp} can be viewed as an $h$-cobordism version of a result about pseudoisotopies, due to Hatcher \cite[Appendix A]{Hat78}; see also \cite[Appendix B]{Munoz-Echaniz} for a related result.
    \end{RR}
    
    We first prove Theorem \ref{thm: prod decomp} in the case $P$ is an interval. We then prove Theorem \ref{thm: prod decomp} in the case $P$ is a higher-dimensional disc, assuming Theorem \ref{thm: prod decomp} holds in strictly lower dimension. 
    
    To prove Theorem \ref{thm: prod decomp} in full generality, we perform a double-induction, on the dimension $n$ and the number of handles in a handle decomposition of $P$. Together, this allows us to reduce the general case to the special cases previously considered.

    Though in our applications it suffices to apply Theorem \ref{thm: prod decomp} in the case $P$ has no boundary, for the induction step in the proof to work we must also consider the case where $P$ does have boundary.
    
    For notational simplicity, we write the proof in the case $B$ is a point, taking care that each step can be performed parametrically in $B$.

\subsection{Interval case}
    In this subsection, we show:
    \begin{PP}\label{prop: int case}
        Theorem \ref{thm: prod decomp} holds in the case $P := [0,1]$ is an interval.
    \end{PP}

    Here $R = (W \times \{0,1\}) \cup_{M \times \{0,1\}} (M \times [0,1]) \subseteq W \times [0,1]$. Explicitly, we produce a diffeomorphism:
    \begin{equation}\label{eq: goal interval case}
        W \times P \cong (R \times I) \cup_{M \times P' \times \{1\}} U
    \end{equation}
    relative to $R$, where $P' \subseteq P^\circ$ is a subinterval, and $U$ is an $h$-cobordism on $M \times P'$ which is inverse to $W \times P'$.
    \begin{RR}
        Throughout this section, we encounter many copies of $[0,1]$ or other closed intervals: we name each individually to keep track of the different copies. For an interval $J$, we write $\partial_l J$ and $\partial_r J$ for $\min(J)$ and $\max(J)$ respectively.
    \end{RR}
    
    Let $N = \partial W \setminus M^\circ$, the other boundary end of $W$. Pick a classical right inverse $(W^c;N,M)$ and a diffeomorphism
rel boundary
    \begin{equation}
        W \cup_N W^c \cong M \times J.
    \end{equation}

    We decompose $P$ into subintervals: let $I_1 = [0,\frac13]$ , $I_2 = [\frac13,\frac23]$ , $I_3=[\frac23,1] \subseteq P$, so $P=I_1 \cup I_2 \cup I_3$. Choosing a collar neighbourhood $M \times I \subseteq W$ of $M$ in $W$ and identifying $M \times I \cup_M W \cong W$, we may now write $W \times P$ as 
    \begin{equation}\label{eq: WxP decomp}
        W \times P = (M \times I \times P) \cup (W \times I_1) \cup (W \times I_2) \cup (W \times I_3)
    \end{equation}
    Let $S = (M \times I \times P) \cup (W \times I_1) \cup (W \times I_3)$; this is an $h$-cobordism on $R$ (after smoothing corners). Note this is (\ref{eq: WxP decomp}) with the third term removed.

    After smoothing, $S$ is a collar neighbourhood of the boundary component $R$, and so we obtain:
    \begin{LL}
        $S$ is a trivial $h$-cobordism on $R$.
    \end{LL}
    As depicted in Figure \ref{fig: WxP decomp}, from (\ref{eq: WxP decomp}) we obtain a diffeomorphism relative to $R$: 
    \begin{equation}\label{eq: WxP decomp 2}
        W \times P \cong \left(R \times I \right)\bigcup\limits_{(M \times I_2) \cup (W \times \partial I_2)} \left(W \times I_2\right)
    \end{equation}
    \begin{figure}[h]  
        \centering
        \begin{tikzpicture}[scale=.8,
        v/.style={draw,shape=circle, fill=black, minimum size=1.3mm, inner sep=0pt, outer sep=0pt},
        vred/.style={draw,shape=circle, fill=red, minimum size=1mm, inner sep=0pt, outer sep=0pt},
        vsmall/.style={draw,shape=circle, fill=black, minimum size=1mm, inner sep=0pt, outer sep=0pt}]

            \node at (-1.5, 0.2) {$W \times P \cong$};

            \draw[fill = yellow] (0, 1.2) rectangle (2, 0);
            \draw[fill = yellow] (4, 1.2) rectangle (6, 0);
            \draw[fill = yellow] (0, 0) rectangle (6, -0.8);
            
            \draw (0,0) to (6,0);
            \draw[very thick, red] (0, -0.8) to (6, -0.8);
            \draw (0, 1.2) to (6, 1.2);

            \draw[very thick, blue] (4, 0) to (6, 0);

            \draw[very thick, red] (0, 1.2) to (0, -0.8);
            \draw (2, 1.2) to (2, 0);
            \draw (4, 1.2) to (4, 0);
            \draw[very thick, red] (6, 1.2) to (6, -0.8);

            \node at (3, -0.4) {\footnotesize $M \times P \times I$};

            \node at (1, 0.6) {\footnotesize $W \times I_1$};
            \node at (3, 0.6) {\footnotesize $W \times I_2$};
            \node at (5, 0.6) {\footnotesize $W \times I_3$};

            \begin{scope}[shift = {(0, -0.3)}]
                \draw[->] (6.5, -0.4) to (6, -0.4);
                \node at (6.7, -0.4) {{\color{red} $R$}};
            \end{scope}

            \begin{scope}[shift = {(0, 0)}]
                \draw (6.5, 0.2) to (5.6, 0.2);
                \draw[->] (5.6, 0.2) to (5.6, 0);
                \node at (7.2, 0.1) {{\color{blue}$M \times I_3$}};
            \end{scope}
            
            \begin{scope}[shift = {(0, 0.1)}]
                \draw[->] (6.5, 0.8) to (5.8, 0.8);
                \fill[yellow] (6.5, 1.1) rectangle (7, 0.5);
                \node at (7.6, 0.8) {$S \cong R \times I$};
            \end{scope}

        \end{tikzpicture}
        \caption{Decomposition of $W \times P$ in (\ref{eq: WxP decomp}) and (\ref{eq: WxP decomp 2}).}
        \label{fig: WxP decomp}
    \end{figure}
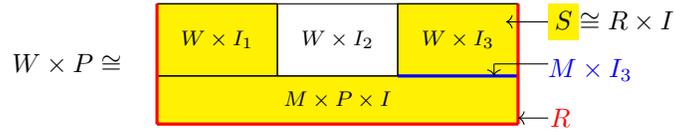
    
    We decompose $I_2$ into subintervals: let $J_1 = [\frac13,\frac25]$, $I_4=[\frac25,\frac35]$, $J_2 = [\frac35,\frac23] \subseteq I_2$. Fix diffeomorphisms $J \cong J_1$ and $J \cong J_2$, where the first reverses orientation and the second is orientation-preserving.

    These induce embeddings $W^c \hookrightarrow M \times J_1$ and $W^c \hookrightarrow M \times J_2$ sending $W$ to $M \times \{\frac13\}$ and $M \times \{\frac23\}$ respectively; let $W^c_1, W^c_2$ be their images and $W_1, W_2$ the images of their complements, respectively; see Figure \ref{fig: MxI2 decomp}. 

    \begin{figure}[h]  
        \centering
        \begin{tikzpicture}[scale=.8,
        v/.style={draw,shape=circle, fill=black, minimum size=1.3mm, inner sep=0pt, outer sep=0pt},
        vred/.style={draw,shape=circle, fill=red, minimum size=1mm, inner sep=0pt, outer sep=0pt},
        vsmall/.style={draw,shape=circle, fill=black, minimum size=1mm, inner sep=0pt, outer sep=0pt}]
            \node at (-1.2, 0) {$M \times I_2 \cong$};

            \draw (0, 0.6) rectangle (8, -0.6);

            \draw (1.5, 0.6) to (1.5, -0.6);
            \draw (3, 0.6) to (3, -0.6);
            \draw (5, 0.6) to (5, -0.6);
            \draw (6.5, 0.6) to (6.5, -0.6);

            \node at (0.75, 0) {\footnotesize $W^c_1$};
            \node at (2.25, 0) {\footnotesize $W_1$};
            \node at (4, 0) {\footnotesize $M \times I_4$};
            \node at (5.75, 0) {\footnotesize $W_2$};
            \node at (7.25, 0) {\footnotesize $W^c_2$};
        \end{tikzpicture}
        \caption{Decomposition of $M \times I_2$.}
        \label{fig: MxI2 decomp}
    \end{figure}

    In terms of these new decompositions, in (\ref{eq: WxP decomp 2}) $W \times I_2$ is glued to $R \times I$ along the subspace:
    \begin{equation}\label{eq: gluing region for WxP decomp}
        \left[W \times \left\{\sfrac13\right\} \cup W^c_1\right] \cup W_1 \cup (M \times I_4) \cup W_2 \cup \left[W^c_2 \cup W \times \left\{\sfrac23\right\}\right] \subseteq R \times \{1\}
    \end{equation}

    Consider the subspaces in square brackets $W \times \{\frac13\} \cup W^c_1$ and $W^c_2 \cup W \times \{\frac23\}$ of (\ref{eq: gluing region for WxP decomp}). Both are diffeomorphic to $M \times J$ (relative to the copies of $M$ in the boundary), so by following the flow of a vector field in the $J$-direction, we can isotope the gluing region (\ref{eq: gluing region for WxP decomp}) to live inside $M \times I_2$.

    Applying Lemma \ref{lem: isotopy diffeo}, we obtain a diffeomorphism relative to $R$:
    \begin{equation}\label{eq: WxP decomp 3}
        W \times P \cong R \times I \bigcup\limits_{W_1 \cup (M \times I_4) \times W_2} W \times I_4
    \end{equation}
    Let $U$ be the $h$-cobordism on $M \times I_2$ shown in Figure \ref{fig: hcob U}, defined by
    \begin{equation}\label{eq: hcob U}
        U = \left(M \times I_2 \times I\right) \bigcup\limits_{(W_1 \cup M \times I_4 \cup W_2) \times \{1\}} W \times I_4
    \end{equation}
    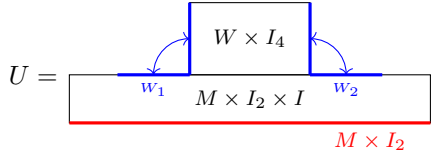
\begin{figure}[h]
        \centering
        \begin{tikzpicture}[scale=.8,
        v/.style={draw,shape=circle, fill=black, minimum size=1.3mm, inner sep=0pt, outer sep=0pt},
        vred/.style={draw,shape=circle, fill=red, minimum size=1mm, inner sep=0pt, outer sep=0pt},
        vsmall/.style={draw,shape=circle, fill=black, minimum size=1mm, inner sep=0pt, outer sep=0pt}]
            \node at (-0.6, 0) {$U =$};
            \draw (0,0) rectangle (6, -0.8);
            \draw (2,0) rectangle (4, 1.2);

            \draw[very thick, blue] (0.8, 0) to (2, 0);
            \draw[very thick, blue] (2, 0) to (2, 1.2);

            \draw[very thick, blue] (4, 0) to (5.2, 0);
            \draw[very thick, blue] (4, 0) to (4, 1.2);

            \draw[very thick, red] (0, -0.8) to (6, -0.8);

            \node at (3, -0.4) {\footnotesize $M \times I_2 \times I$};
            
            \node at (1.4, -0.25) {\tiny {\color{blue} $W_1$}};
            \node at (4.6, -0.25) {\tiny {\color{blue} $W_2$}};

            \node at (3, 0.6) {\footnotesize $W \times I_4$};

            \draw[blue,<->] (2,0.6) arc (90:180:0.6);
            \draw[blue,<->] (4.6,0) arc (0:90:0.6);

            \node at (5,-1.1) {\footnotesize \color{red} $M \times I_2$};
            
        \end{tikzpicture}
        \caption{$U$, an $h$-cobordism on $M \times I_2$}
        \label{fig: hcob U}
    \end{figure}
    
    Note that despite the similarity in definitions, this is not necessarily isomorphic to the $h$-cobordism $W \times I_4$, since we glue along a subspace which isn't just $M \times I_4$: it is for this reason we must keep careful track of the regions along which we glue manifolds.

    Setting $P' := I_2$, we obtain a diffeomorphism of the form (\ref{eq: goal interval case}); it then suffices to show:

    \begin{LL}
        $U$ is an inverse $h$-cobordism to $W \times P'$. 
        
        Equivalently, $U + (W \times P') = 0$ in $H_{pt}(M \times P')$.
    \end{LL}
    \begin{proof}
        Similarly to (\ref{eq: gluing region for WxP decomp}), we may decompose $M \times P'$ as 
        \begin{equation}\label{eq: MxP decomp}
            M \times P' \cong W^c_1 \cup_N W_1 \cup_M (M \times K_1) \cup_M W_2 \cup_N W^c_2 \cup_M (M \times K_2) \cup_M (M \times K_3)
        \end{equation}
        for appropriate subintervals $K_1, K_2, K_3 \subseteq P'$, and $W_i$, $W^c_i$ copies of $W$ and $W^c$ respectively.

        Then by definition, $U + (W \times P')$ is the $h$-cobordism on $M \times P'$ described as follows:
        \begin{equation}\label{eq: UpWxI decomp}
            U + (W \times P') = (M \times P' \times I) \cup (W \times K_1) \cup (W \times K_2)
        \end{equation}
        where we glue the second and third bracketed expressions to the first as follows (see Figure \ref{fig: UWP decomp}). We glue $W \times K_1$ to $M \times P' \times I$ by identifying $(W \times \partial K_1) \cup (M \times K_1 \times \{1\}) \subseteq W \times K_1$ with $(W_1 \cup (M \times K_1) \cup W_2)\subseteq$ (\ref{eq: MxP decomp}), and we glue $W \times K_2$ to $M \times P' \times I$ by identifying $M \times K_2$ with the $(M \times K_2) \subseteq$ (\ref{eq: MxP decomp}).

        \begin{figure}[h]
            \centering
            \begin{tikzpicture}[scale=.8,
            v/.style={draw,shape=circle, fill=black, minimum size=1.3mm, inner sep=0pt, outer sep=0pt},
            vred/.style={draw,shape=circle, fill=red, minimum size=1mm, inner sep=0pt, outer sep=0pt},
            vsmall/.style={draw,shape=circle, fill=black, minimum size=1mm, inner sep=0pt, outer sep=0pt}]
                \node at (-1.8, 0) {$U+(W \times P') =$};
                \draw (0,0) rectangle (8.6, -0.8);
                \draw (1.6, 0) rectangle (3.6, 1.2);
                \draw (5.6, 0) rectangle (7.6, 1.2);
    
                \draw[very thick, blue] (0.6, 0) to (1.6, 0);
                \draw[very thick, blue] (1.6, 0) to (1.6, 1.2);
    
                \draw[very thick, blue] (3.6, 0) to (4.6, 0);
                \draw[very thick, blue] (3.6, 0) to (3.6, 1.2);
    
                \draw[very thick, orange] (4.6, 0) to (5.6, 0);
                \draw[very thick, orange] (5.6, 0) to (5.6, 1.2);
    
                \draw[very thick, red] (0, -0.8) to (8.6, -0.8);
    
                \node at (5.8, -0.4) {\footnotesize $M \times P' \times I$};
                
                \node at (1.2, -0.25) {\tiny {\color{blue} $W_1$}};
                \node at (4.1, -0.25) {\tiny {\color{blue} $W_2$}};
    
                \node at (2.6, 0.6) {\footnotesize $W \times K_1$};
                \node at (6.6, 0.6) {\footnotesize $W \times K_2$};
    
                \node at (5, 0.6) {\footnotesize \color{orange} $C$};
                \draw[->] (5.2, 0.6) to (5.6, 0.6);
    
                \draw[blue,<->] (1.6,0.6) arc (90:180:0.6);
                \draw[blue,<->] (4.2,0) arc (0:90:0.6);
    
                \node at (4.3,-1.1) {\footnotesize \color{red} $M \times P'$};
                
            \end{tikzpicture}
            \caption{$U+(W \times P')$, an $h$-cobordism on $M \times P'$}
            \label{fig: UWP decomp}
        \end{figure}
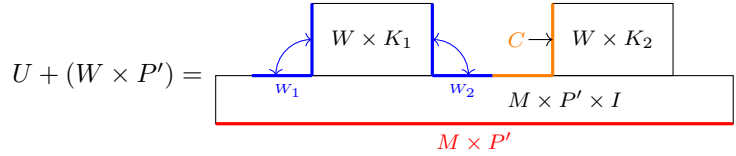

        Consider the subspace $C := W^c_2 \cup (W \times \partial_l K_2) \subseteq \partial(U + (W \times P'))$ shown in Figure \ref{fig: UWP decomp} in {\color{orange} orange}. This is diffeomorphic to $N \times J$, so we may isotope the $4^{\textrm{th}}$ term of (\ref{eq: MxP decomp}) to $W \times \partial_l K_2$ (lying in the $3^{\textrm{rd}}$ term of (\ref{eq: UpWxI decomp})). Therefore by Lemma \ref{lem: isotopy diffeo} applied to this isotopy, we find that:
        \begin{equation}\label{eq: decomp -}
            U+(W \times P') \cong (M \times P' \times I) \cup (W \times P'')
        \end{equation}
        where $P'' \subseteq (P')^\circ$ is a subinterval. Here we glue $W \times P''$ to $M \times P'$ along $W \cup (M \times P'')$, where $W \hookrightarrow M \times [\partial_l P', \partial_l P'']$ is a choice of embedding sending $M$ to $M \times \partial_l P''$; this is similar to Figure \ref{fig: hcob U} except we only glue along one side of $W \times \textrm{(interval)}$ instead of both.

        Since we glue along a region diffeomorphic to $W \cup_M M \times [0,1]$, we find that
         \begin{equation}
            U+(W \times P') \cong (M \times P' \times I) \cup (W \times P'')
        \end{equation}
        where here we glue $W \times \partial_l P'' \subseteq W \times P''$ to $M \times P' \times I$ along some embedding $W \hookrightarrow (M \times P' \times \{1\})^\circ$. Since $W \times P''$ is a trivial $h$-cobordism when viewed as an $h$-cobordism on $W \times \partial_l P''$, we find that $U + (W \times P') \cong M \times P' \times I$, as required.
    \end{proof}

\subsection{Even-dimensional disc case}
    In this section, we show:
    \begin{PP}\label{prop: even case}
        Let $n \geq 2$ be even. Assume that the statement of Theorem \ref{thm: prod decomp} holds for all manifolds of dimension strictly less than $n$. 

        Then Theorem \ref{thm: prod decomp} holds for $P := D^n$.
    \end{PP}
    \begin{proof}
        Decompose $D^n$ as
        \begin{equation}\label{eq: decomp 1}
            D^n = D^n(\eps) \cup (S^{n-1} \times J_1) \cup (S^{n-1} \times J_2)
        \end{equation}
        where $D^n(\eps)$ is a disc inside $D^n$ of smaller radius, and $J_1$, $J_2$ are appropriate intervals; see Figure \ref{fig: Dn decomp}. In particular, $S^{n-1} \times J_2$ is a collar for $\partial D^n$.
        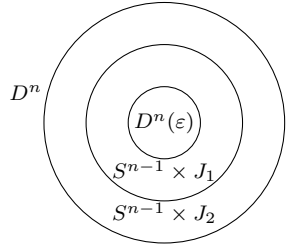
\begin{figure}[h]
            \centering
            \begin{tikzpicture}[scale=.8,
            v/.style={draw,shape=circle, fill=black, minimum size=1.3mm, inner sep=0pt, outer sep=0pt},
            vred/.style={draw,shape=circle, fill=red, minimum size=1mm, inner sep=0pt, outer sep=0pt},
            vsmall/.style={draw,shape=circle, fill=black, minimum size=1mm, inner sep=0pt, outer sep=0pt}]
                \node at (-2.3, 0.5) {\footnotesize $D^n$};
                \node at (0,0) {\footnotesize $D^n(\eps)$};
                \node at (0, -0.8) {\footnotesize $S^{n-1} \times J_1$};
                \node at (0, -1.5) {\footnotesize $S^{n-1} \times J_2$};

                \draw (0.6,0) arc (0:360:0.6);
                \draw (1.3,0) arc (0:360:1.3);
                \draw (2,0) arc (0:360:2);
                
            \end{tikzpicture}
            \caption{Decomposition of $D^n$.}
            \label{fig: Dn decomp}
        \end{figure}
    
        Using $W \cong (M \times I) \cup W$, we obtain a decomposition of $W \times D^n$:
        \begin{equation}\label{eq: odd disc decomp 1}
            W \times D^n \cong (M \times I \times D^n) \cup (W \times D^n(\eps)) \cup (W \times S^{n-1} \times J_1) \cup (W \times S^{n-1} \times J_2)
        \end{equation}
        Let $S = (M \times I \times D^n) \cup (W \times D^n(\eps)) \cup (W \times S^{n-1} \times J_2) \subseteq W \times D^n$. In (\ref{eq: odd disc decomp 1}), $(W \times S^{n-1} \times J_1)$ is glued onto $S$ along $(M \times S^{n-1} \times J_1) \cup (W \times S^{n-1} \times \partial J_1)$, and is an $h$-cobordism on this this subspace. However, this $h$-cobordism is of the form $W' \times S^{n-1}$ for an $h$-cobordism $W'$, so by applying Theorem \ref{thm: prod decomp} for $P=S^{n-1}$ (where we have assumed it holds), this is a trivial $h$-cobordism. 
    
        In particular, we obtain a diffeomorphism (relative to $R$):
        \begin{equation}
            W \times D^n \cong (M \times I \times D^n) \cup (W \times D^n(\eps)) \cup (W \times S^{n-1} \times J_2)
        \end{equation}
        as required.
    \end{proof}
\subsection{Odd-dimensional disc case}
    In this section, we show:
    \begin{PP}\label{prop: odd case}
        Let $n \geq 3$ be odd. Assume that the statement of Theorem \ref{thm: prod decomp} holds for all manifolds of dimension strictly less than $n$. 

        Then Theorem \ref{thm: prod decomp} holds for $P := D^n$.
    \end{PP}
    \begin{proof}
        Choose a decomposition of $D^n$ as in (\ref{eq: decomp 1}), and let $V=W \times J_1$, viewed as an $h$-cobordism on $(M \times J_1) \cup (W \times \partial J_1)$.

        Let $D^{n-1}_\pm \subseteq S^{n-1}$ be neighbourhoods of two antipodal points $p_{\pm} \in S^{n-1}$. Let $T = D^n(\eps) \cup (D^{n-1}_+\times J_1) \cup (D_-^{n-1} \times J_1) \subseteq D^n$ (see Figure \ref{fig: cross}).

        Applying Theorem \ref{thm: prod decomp} to $V$ and $S^{n-1}$, we obtain a diffeomorphism (relative to $R$):
        \begin{equation}\label{eq: decomp odd case}
            W \times P \cong (M \times I \times D^n) \cup (W \times T) \cup (W \times S^{n-1} \times J_2)
        \end{equation}
        Let $K$ be a line segment in the disc $(S^{n-1} \times J_1) \cup D^n(\eps)$ between the two points $p_\pm \times \partial_+J_1$ (so $K$ is diffeomorphic to an interval).

        Then as shown in Figure \ref{fig: cross}, $T$ is diffeomorphic to $D^{n-1} \times K$, via a diffeomorphism restricting to the natural diffeomorphism between their subspaces $(D^{n-1}_+ \times \partial_r J_1) \cup (D^{n-1}_- \times \partial_r J_1) \cong D^{n-1} \times \partial K$.

        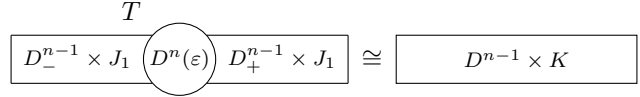
\begin{figure}[h]  
            \centering
            \begin{tikzpicture}[scale=.8,
            v/.style={draw,shape=circle, fill=black, minimum size=1.3mm, inner sep=0pt, outer sep=0pt},
            vred/.style={draw,shape=circle, fill=red, minimum size=1mm, inner sep=0pt, outer sep=0pt},
            vsmall/.style={draw,shape=circle, fill=black, minimum size=1mm, inner sep=0pt, outer sep=0pt}]
    
                \node at (0,0) {\footnotesize $D^n(\eps)$};
                \draw (0.6,0) arc (0:360:0.6);

                \draw (0.45, 0.4) to (2.8, 0.4);
                \draw (0.45, -0.4) to (2.8, -0.4);
                \draw (2.8, 0.4) to (2.8, -0.4);

                \draw (-0.45, 0.4) to (-2.8, 0.4);
                \draw (-0.45, -0.4) to (-2.8, -0.4);
                \draw (-2.8, 0.4) to (-2.8, -0.4);

                \node at (1.7, 0) {\footnotesize $D^{n-1}_+ \times J_1$};
                \node at (-1.7, 0) {\footnotesize $D^{n-1}_- \times J_1$};

                \node at (-0.8, 0.8) {$T$};

                \node at (3.2, 0) {$\cong$};

                \draw (3.6, 0.4) rectangle (7.6, -0.4);

                \node at (5.6, 0) {\footnotesize $D^{n-1} \times K$};

            \end{tikzpicture}
            \caption{Diffeomorphism $T \cong D^{n-1} \times K$}
            \label{fig: cross}
        \end{figure}

        Combining with (\ref{eq: decomp odd case}), we obtain a diffeomorphism relative to $R$:

        \begin{equation}\label{eq: decomp odd case 2}
            W \times P \cong (M \times I \times D^n) \cup (W \times D^{n-1} \times K) \cup (W \times S^{n-1} \times J_2)
        \end{equation}

        Applying Proposition \ref{prop: int case} (explicitly, the diffeomorphism (\ref{eq: goal interval case}), we obtain a diffeomorphism:
        \begin{equation}\label{eq: decomp odd case 3}
            W \times D^{n-1} \times K \cong \left[ (M \times I \times K) \cup (W \times I \times \partial K) \cup U\right] \times D^{n-1}
        \end{equation}
        relative to $(M \times D^{n-1} \times K) \cup (W \times D^{n-1} \times \partial K)$, where $U$ is an $h$-cobordism on $M \times K'$ inverse to $W \times K'$ (and $K' \subseteq K^\circ$ is a subinterval). In particular, $U \times D^{n-1}$ is an inverse to $W \times K' \times D^{n-1} \cong W \times D^n$, i.e. it represents $-(W \times D^n)$.

        Extending (\ref{eq: decomp odd case 3}) to the entirety of (\ref{eq: decomp odd case 2}) by the identity, we obtain a diffeomorphism relative to $R$:
        \begin{equation}
            W \times P \cong (M \times I \times D^n) \cup (W \times S^{n-1} \times J_2) \cup \left( -W \times D^n(\eps) \right)
        \end{equation}
        as required.
    \end{proof}

\subsection{General case}
    \begin{proof}[Proof of Theorem \ref{thm: prod decomp}]
        By induction on the dimension $n$ of $P$, we may assume the statement of Theorem \ref{thm: prod decomp} holds for all manifolds $P'$ of dimension $\leq n-1$ (the $n=0$ case is trivial).

        Choose a handle decomposition of $P$; let $s$ be the number of handles. Note that Propositions \ref{prop: int case}, \ref{prop: even case} and \ref{prop: odd case} prove the case where there is exactly one handle. By induction on the number of handles $s$, we may assume the statement of Theorem \ref{thm: prod decomp} holds for all manifolds $P$ of dimension $n$ which admit handle decompositions with $\leq s-1$ handles.

        We may write $P = P' \cup H^n_i$, where $H^n_i = D^i \times D^{n-i}$ is an $n$-dimensional $i$-handle, glued to $P'$ along some embedding $S^{i-1} \times D^{n-i} \to \partial P'$. By the above discussion, we may assume Theorem \ref{thm: prod decomp} holds for $P'$. Let $\chi = \chi(P, \partial P)$ and $\chi' = \chi(P', \partial P')$, so $\chi = \chi' + (-1)^{n-i}$.

        Applying Theorem \ref{thm: prod decomp} to $P'$, we obtain a diffeomorphism:
        \begin{equation}\label{eq: gen case 1}
            W \times P' \cong (R' \times I) \cup \left(\chi' \cdot (W \times D^n)\right)
        \end{equation}
        relative to $R'=(M \times P') \cup (W \times \partial P')$.

        Applying Theorem \ref{thm: prod decomp} to $D^{n-i}$ and taking products with $D^i$ on both sides, we obtain a diffeomorphism:
        \begin{equation}\label{eq: gen case 2}
            W \times H^n_i \cong (M \times I \times H^n_i) \cup \left( (-1)^{n-i} \cdot (W \times D^n) \right) \cup (W \times I \times D^i \times \partial D^{n-i})
        \end{equation}
        relative to the subspaces $(M \times \{0\} \times H^n_i) \cup (W \times \{0\} \times D^i \times \partial D^{n-i})$ on both sides.

        Extending (\ref{eq: gen case 1}) via the identity over $W \times H^n_i$, and noting that $(R' \times I) \setminus (W \times H^n_i) = (R \times I) \setminus (W \times H^n_i)$ (viewing these spaces as subspaces of $W \times P$), we obtain a diffeomorphism:
        \begin{equation}\label{eq: gen case 3}
            W \times P \cong (R \times I) \cup (\chi' \cdot (W \times D^n)) \cup (W \times H^n_i)
        \end{equation}
        relative to $R$. Since $R \cap (W \times H^n_i) = W \times D^i \times \partial D^{n-i}$, we may apply apply (\ref{eq: gen case 2}) to (\ref{eq: gen case 3}) (and apply the identity over the rest of the right hand side) to obtain a diffeomorphism:
        \begin{equation}
            W \times P \cong (R \times I) \cup (\chi' \cdot (W \times D^n)) \cup \left( (-1)^{n-i} \cdot (W \times D^n) \right)
        \end{equation}
        relative to $R$; since $\chi=\chi' + (-1)^{n-i}$ this is the desired isomorphism.
    \end{proof}

\section{Nearby Lagrangians and parametrised Whitehead torsion}\label{sec:nearbylag}
\subsection{Nearby Lagrangians and their Legendrian lifts}

     Given a manifold $X$, we write $\lambda_X$ for the tautological Liouville $1$-form on $T^*X$ and use the contact $1$-form $dz-\lambda_X$ on $J^1(X)=T^*X\times \R$.
\begin{DD}
For $M$ a closed connected manifold and $B$ a compact manifold.
A \emph{family of nearby Lagrangian submanifolds} in $T^* M$ parametrised by $B$ is a family $\bL={L_b}_{b\in B}$
of closed connected exact Lagrangian submanifolds
of $T^* M$. More explicitly, $\bL \to B$ is a smooth subbundle of the trivial bundle $T^*M \times B \to B$,
such that each fibre $L_b$ is a closed exact Lagrangian in $T^*M$. Exactness means that for each $b\in B$, there exists a smooth function $f_b$ on $L_b$ such that
$i_b^*\lambda_M=d f_b$ where $i_b:L_b\to T^* M$ is the inclusion.

Two such families are \emph{isotopic} if there exists a family over $B\times [0,1]$ which restricts to the given families
over $\{0,1\}$. We denote $\cL_B(M)$ the set of isotopy classes of families of nearby Lagrangian submanifolds parametrised by $B$.
\end{DD}

It is convenient, especially when considering generating functions, to lift exact Lagrangian submanifolds
of $T^*M$ to Legendrian submanifolds in $J^1(M)$. This amounts to a choice of a primitive function $f_b$
as in the above definition, and in our case this can be chosen smoothly in $b$.

\begin{LL}
Let $(L_b)_{b\in B}$ be a family of nearby Lagrangian submanifolds. There exists a smooth function $f$
on the total space such that for each $b \in B$, $i_b^*\lambda_M= df_b$.
\end{LL}
\begin{proof}
The function $f$ can be found locally near each $b$, but since all such primitives differ by a constant
($L_b$ is connected) we can globalise the construction and find the required smooth function on the total space
of the family.
\end{proof}

A function $f$ as provided by the previous lemma allows to view $(L_b)_{b\in B}$
as a family $(\widetilde{L_b})_{b\in B}$ of Legendrian submanifolds in $J^1(M)$ but such
a family is canonically the same thing as a single Legendrian submanifold (the total space)
in $J^1(M\times B)$ in view of the following lemma.

    \begin{LL}\label{LL:biglegendrian}
    Let $B$, $L$ and $M$ be manifolds.
    \begin{enumerate}
        \item If $j:L\to J^1(M\times B)$ is a Legendrian immersion such that the projection $\pi:L\to B$ is a submersion, then for each $b\in B$, the restriction of the projection $i:L\to J^1(M)$ to $\pi^{-1}(b)$ is a Legendrian immersion.
        \item If $\pi: L\to B$ is a submersion and $i:L\to J^1(M)$ a map which restricts to a Legendrian immersion $\pi^{-1}(b) \to J^1(M)$ for all $b\in B$, then there is a unique Legendrian immersion $j:L\to J^1(M\times B)$ which lifts $(i,\pi):L\to J^1(M)\times B$.
    \end{enumerate}
    \end{LL}
    \begin{proof}
    \begin{enumerate}
        \item We have $j^*(dz-\lambda_M -\lambda_B)=0$ on $TL$ and $j^*\lambda_B=0$ on $\ker d\pi$, so $j^*(dz -\lambda_M)=0$ on $\ker d\pi$. Hence $i:\pi^{-1}(b)\to J^1(M)$ is a Legendrian map for each $b\in B$.

        Furthermore $TL$ is transverse to each coisotropic subspace $T(T^* M \times T^*_b B)$, so the symplectic reduction map $\ker d \pi = T(\pi^{-1}(b))=TL \cap T(T^* M \times T^*_b B) \to T(T^* M)$ is injective and we conclude that $\pi^{-1}(b)\to J^1(M)$ is an immersion.

        \item The $1$-form $i^*(dz-\lambda_M)$ vanishes on $\ker d\pi$, so there is a unique map $\beta: L \to T^* B$ lifting $\pi$ (i.e. a section of $\pi^*T^*B\to L$) such that $\beta^*\lambda_B=i^*(dz-\lambda_M)$ or equivalently, such that the map $j=(i,\beta):L\to J^1 M\times T^* B = J^1(M\times B)$ satisfies 
        \[j^* (dz-\lambda_M - \lambda_B)=i^*(dz-\lambda_M)-\beta^* \lambda_B=0.\]
        The map $j$ is an immersion since $(i,\pi):L\to J^1(M)\times B$ is already an immersion.
    \end{enumerate}
    \end{proof}

\begin{DD}
We define $\cL(M)$ to be the (geometric realisation of the) simplicial set whose $k$-simplices are families of nearby Lagrangian
submanifolds parametrised by a standard simplex $\Delta^k$, with face and degeneracy maps given by pullbacks.
\end{DD}

\begin{PP}\label{prop:classifyingnearby}
For any compact manifold $B$, there is a canonical bijection of sets
\[\cL_B(M)\to [B,\cL(M)].\]
\end{PP}

\begin{RR}
We could also define $\cL(M)$ as the space of closed exact Lagrangian submanifolds with the $C^\infty$-topology
and we would get a (weakly) homotopy equivalent space but it is more convenient for us to define $\cL(M)$ as a simplicial
set as $\cH(M)$ is also defined this way.
\end{RR}

\subsection{Rephrasing Theorem~\ref{thm:main}}

    Let $M$ be a closed manifold of dimension $n$. Let $B$ be a compact manifold with boundary of dimension $d$. Let $\bL = \{L_b\}_{b \in B}$ be a smooth family of nearby Lagrangians in $T^*M$, parametrised by $B$. Let $\{g_b: L_b \to M\}_b$ be the canonical family of projections. In the next section, we will prove:

    \begin{TT}\label{thm:mainoverB}
        Assume $l \gg 0$. Then there is a smooth family of disk vector bundles $\{E_b \to L_b\}_b$ over $\{L_b\}_b$, and a smooth family of smooth $h$-cobordisms $\bX = \{X_b\}_b$ on $D^{l-1}$, such that there is a smooth family of diffeomorphisms
        \begin{equation}\label{eq: thm 412}
            \left\{ E_b \cong M \times (D^l \cup X_b)\right\}
        \end{equation}
        On the right-hand side, we glue along some fixed choice of embedding $i: D^{l-1} \hookrightarrow \partial D^l$. Furthermore the following diagram commutes up to homotopy (continuously in $b$):
        \begin{equation}\label{eq:mainoverB}
            \begin{tikzcd}
                E_b
                \arrow[r]
                \arrow[d]
                &
                M \times (D^l \cup X_b)
                \arrow[d]
                \\
                L_b 
                \arrow[r]
                &
                M
            \end{tikzcd}
        \end{equation}
    \end{TT}

    \begin{RR}
        The obstruction to being able to do this with trivial disc bundles, so $E_b = L_b \times \bR^l$, is that the virtual vector bundle $g_b^*TM-TL_b$ could be stably nontrivial. 
        
        The (Strong) Nearby Lagrangian conjecture predicts that this bundle is stably trivial, but this has not yet been proved.
    \end{RR}

    \begin{proof}[Proof of Theorem~\ref{thm:main} from Theorem \ref{thm:mainoverB}]
According to Remark~\ref{rem:CWorMFD} and Proposition~\ref{prop:classifyingnearby} it is enough
to treat the case of a smooth family of nearby Lagrangian submanifolds $(L_b)_{b\in B}$ for $B$ a compact manifold. 
According to Definition~\ref{def:paramWh}, the parametrised Whitehead torsion $w((L_b)_{b\in B})$
is represented by the family of $h$-cobordisms $(E_b\times I,M\times D^l\times \{0\})$, which is obtained from the trivial
$h$-cobordism $M\times D^l\times I$ by attaching $M\times X_b\times I$, an $h$-cobordism on $M\times D^{l-1}\times I$.
Hence the family $(X_b\times I)_b$ represents a map $\delta:B\to H(D^{l-1}\times I)\to\cH(\pt)$ such that $p\circ \delta=w((L_b)_{b\in B})$.
\end{proof}

\section{Generating functions and the factorisation through $\cH(\pt)$}\label{sec:genfun}

    \subsection{Generating functions of tube type and their difference functions}
        \begin{DD}\label{def:2}
            A smooth function $f: \bR^l \to \bR$ is \emph{almost quadratic} if it is of the form  $f=q+g$, where $q$ and $g$ are $C^1$ and satisfy:
            \begin{enumerate}
                \item $q$ is a degree 2 homogeneous function (meaning $q(\lambda v) = \lambda^2 q(v)$ for $\lambda \geq 0$ and $v\in \R^l$),
        	\item $\nabla q(v) \neq 0$ for all $v\neq 0$,
                \item $|\nabla g(v)|\leq c$ for some constant $c>0$.
            \end{enumerate}

A function is of \emph{tube type} if it is almost quadratic and homotopic to a non-degenerate quadratic form through almost quadratic functions.
        \end{DD}

\begin{RR}
    We do note require $q$ and $g$ themselves to be smooth: typically they will not be smooth at 0.
\end{RR}
\begin{RR}\label{rem:compactsupp}
Without changing the critical points of $f=q+g$, we can cut-off $g$ to make it compactly supported.
However the compactly supported condition is not well-behaved with respect to direct sum, that is why
we use the relaxed condition that $|\nabla g|$ is bounded.
\end{RR}

\begin{DD}
A function $f:M\times \R^l\to \R$ is a \emph{generating function of tube type} if each restriction $f:\{x\} \times \R^l \to \R$ is of tube type
and the map $\frac{\del f}{\del v}:M\times \R^l \to \R^l$ vanishes transversely.
\end{DD}

Such a function $f$ generates a Legendrian immersion in $J^1(M)$ by the assignment 
\[(x,v)\in \left\{\frac{\del f}{\del v}(x,v)=0\right\} \mapsto \left(x,\frac{\del f}{\del x}(x,v),f(x,v)\right).\]
If this map is an embedding with image $L$, we say that $L$ is \emph{generated by} $f$.

The generating functions of tube type are generalisations of the more traditional generating functions quadratic at infinity
which correspond to the case where $q$ is a non-degenerate quadratic form.

For Legendrian submanifolds $L$ which lift closed exact Lagrangian submanifolds of $T^* M$,
an existence result for generating functions was proved in \cite{ACGK} (and reformulated in this quadratic setting in \cite{AACK}).
However, the generating functions produced there are in general \emph{twisted} by a cocycle of non-degenerate quadratic forms.
\begin{TT}\label{thm:acgk}\cite{ACGK, AACK}
Let $M$ be a closed manifold, $L$ a closed Legendrian submanifold of $J^1M$ which projects injectively to $T^*M$ and $g_L:L\to M$ the projection.
There exists 
\begin{itemize}
\item a finite totally ordered open cover $(M_i)_{i\in I}$
\item integers $n_i\in \N$, 
\item generating functions of tube type $f_i:M_i\times \R^{n_i}\to \R$ generating $L_i=g_L^{-1}(M_i)$,
\item fiberwise non-degenerate quadratic forms
$q_{ij}:M_i\cap M_j \times \R^{n_j-n_i} \to \R$ for all $i<j$
\end{itemize}
 such that
for all $i<j$, $f_i\oplus q_{ij}=f_j$ on $(M_i\cap M_j) \times \R^{n_j}$ (and in particular $q_{ij}\oplus q_{jk}=q_{ik}$ over $M_i\cap M_j\cap M_k)$.
\end{TT}

\begin{RR}
Technically speaking, in \cite{AACK} the twisted generating function is constructed for a Legendrian submanifold $L'$ which is
Legendrian isotopic to $L$. However, twisted generating functions persist under Legendrian isotopy as demonstrated in \cite[Theorem 3.9]{ACGK} in the
linear at infinity setting (the proof works verbatim in the quadratic setting, see \cite[Remark 3.35]{AACK}).
\end{RR}

In \cite{ACGK} it is proved that the \emph{twisted difference function}  $(U_i,K_i=f_i\oplus (-f_i), q_{ij}\oplus (-q_{ij}))$ can be untwisted (see \cite[Lemma 3.21]{ACGK}), namely there exists fiberwise non-degenerate quadratic form $Q_i$ over $U_i$ such that $K_i\oplus Q_i=K_j\oplus Q_j$ over $U_{ij}$ (this comes from the fact that $q_{ij}\oplus (-q_{ij})$ is nullhomotopic as a map from $U_{ij}$ to the space of non-degenerate quadratic forms). Hence we obtain a global difference function $K:M\times \R^N \to \R$ with the following properties:
\begin{enumerate}
    \item $K$ is Morse-Bott with critical locus diffeomorphic to $L$,
    \item each restriction $K_x:\R^l\to \R$, for $x\in M$, is of tube type.
\end{enumerate}

The parametrised versions of these results can be deduced directly as follows.
Let $B$ be a compact manifold and $L\to B$ a bundle of closed Legendrian submanifolds which project injectively to $T^* M$. In view of Lemma~\ref{LL:biglegendrian}, $L$ lifts uniquely to a Legendrian submanifold $L'\subset J^1(M\times B)$, and the projection $L'\to T^*(M \times B)$ is injective. Next we observe that if $f':M\times B\times \R^l \to \R$ is a generating function for $L'$
then $f'_b:M\times \R^l\to \R$ is a generating function for $L_b$ for each $b\in B$.

The only issue is that $B$ may have boundary so we need one more step to apply Theorem~\ref{thm:acgk}.
Consider the double manifold $N=M\times (B\cup_{\del B} B)$ and the double Legendrian submanifold $L''\subset J^1(N)$
(after smoothing near $\del B$) which projects injectively to $T^*N$. We may then apply Theorem~\ref{thm:acgk}
and restrict everything to $M\times B \subset N$ to obtain the following datum:

\begin{itemize}
    \item a directed open cover $(U_i)_{i\in I}$ of $M\times B$,
    \item a generating function of tube type $f_i:U_i\times \R^{n_i}\to \R$ for $L'$ over $U_i$,
    \item a fiberwise non-degenerate quadratic form $q_{ij}:U_{ij}\times \R^{n_j-n_i} \to \R$ over $U_{ij}$,
\end{itemize}
satisfying for all $i<j$, over $U_i\cap U_j$,
\[f_i\oplus q_{ij}=f_j.\]
For each $b\in B$, the restriction $(M_i^b=U_i\cap (M\times\{b\}), f_i^b, q_{ij}^b)$ is a twisted generating function for $L^b$.
As above, the difference function can be untwisted and therefore we obtain a global function with Morse-Bott critical locus diffeomorphic to $L$.
The properties of this function that we will use are summarised in the following proposition. In fact we will not
need anything else from generating functions to prove our main result.

        \begin{TT}\label{prop:1}
Let $\{L_b\}_{b\in B}$ be a smooth family of nearby Lagrangian submanifolds of $T^* M$. There is a smooth family of functions $\{K_b: M \times \bR^l \to \bR\}_b$ such that:
            \begin{enumerate}
                \item Each $K_b$ is Morse-Bott.
                \item \label{prop1: pt 2} Each critical locus $\operatorname{Crit}(K_b)$ is diffeomorphic to $L_b$, such that the following diagram commutes:
                 \begin{equation*}
                    \begin{tikzcd}
                        \operatorname{Crit}(K_b) 
                        \arrow[rr, "\cong"]
                        \arrow[dr]
                        &&
                        L_b
                        \arrow[dl]
                        \\
                        &M&
                    \end{tikzcd}
                \end{equation*}
                \item The fibrewise negative eigenbundle of $\{K_b\}_b$ along $\{L_b\}$, viewed as a family of vector bundles over $\{L_b\}_b$ (or equivalently as a single vector bundle over the total space $L$), is trivial and of even rank $2k$.
                \item $l \gg n+2k$.
                \item Each $K_b$ is fibrewise of tube type.
            \end{enumerate}
            We call each $K_b$ a \emph{difference function of tube type}.
        \end{TT}
    
\begin{proof}
The construction of $K$ explained above has all properties except maybe for the two last ones.
Let $E_b$ be the (possibly non-trivial) negative eigenbundle of $L_b$ with respect to the Morse-Bott function $K_b$. Since $g_b:L_b\to M$ is a homotopy equivalence,
there are families of vector bundles $E_b'$ and $E_b''$ on $M$ such that $E_b'\oplus E_b'' \simeq M \times \R^m$ and $g_b^* E_b' \oplus E_b$ is trivial
of rank $2k$.
We then define $K':M\times B\times \R^m\times \R^l\to \R$ by the formula
\[K'(x,b,u',u'',v)=K(x,b,v)-|u'|^2+|u''|^2.\]
Then $K_b'$ is still Morse-Bott with critical locus $L_b$ and the negative eigenbundle of $L_b$ with respect to $K'_b$ is now trivial by construction.
The condition $l'=m+l\gg n+2k$ can be achieved by stabilizing $E''_b$ as much as necessary.
\end{proof}

    \subsection{Geography of the constructions}
        In the following sections, we introduce many intermediate spaces and maps; to help keep track of all of these, we summarise them in the following diagram. We warn the reader that the full diagram does \emph{not} commute, though many subdiagrams do (as stated throughout).
        
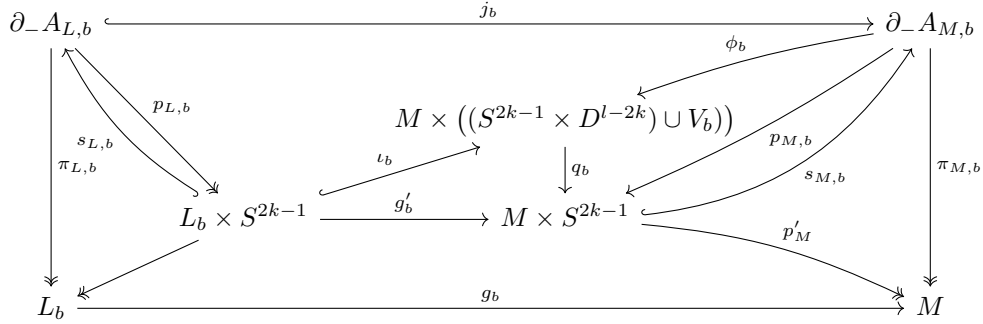
\begin{center}
            \begin{tikzcd}
                \del_- A_{L,b} 
                \arrow[rrrr,"j_b", hookrightarrow] 
                \arrow[ddd,"\pi_{L,b}", twoheadrightarrow] 
                \arrow[rdd,"p_{L,b}", twoheadrightarrow ]
                & & & &
                \del_- A_{M,b} 
                \arrow[ddd,"\pi_{M,b}", twoheadrightarrow ]
                \arrow[ddll,"p_{M,b}", bend left=5] 
                \arrow[dll, bend right=5,swap,"\phi_b"]
                \\ & & 
                M\times \left((S^{2k-1} \times D^{l-2k}) \cup V_b)\right)
                \arrow[d,"q_b"]
                & & \\ & 
                L_b \times S^{2k-1} 
                \ar[ur, "\iota_b", hookrightarrow]
                \arrow[r,"g'_b"] 
                \arrow[ld,twoheadrightarrow]
                \arrow[luu, bend left=15,"s_{L,b}", hookrightarrow] 
                & 
                M\times S^{2k-1}
                \arrow[rrd,"p'_M",bend left= 10,twoheadrightarrow]
                \arrow[rruu,swap,bend right=25,"s_{M,b}", hookrightarrow]
                & & \\
                L_b \arrow[rrrr,"g_b"]
                & & & &
                M
            \end{tikzcd}
        \end{center}
        
    \subsection{Disk bundles}
        
        Let $\{K_b\}_b$ be a smooth family of difference functions of tube type, as in Theorem~\ref{prop:1}. We identify the bundle of critical loci $\{\operatorname{Crit}(K_b)\}_b$ with $\{L_b\}_b$. We fix a Riemannian metric on $M$, and use the standard metric on $\bR^l$; together these induce a metric on $M \times \bR^l$. 

        Let $\lambda\gg 0 $ be large enough that each $K_b$ is fibrewise degree 2 homogeneous outside of the ball $D^l_\lambda$ of radius $\lambda$ (see Remark~\ref{rem:compactsupp}), and let $S^{l-1}_\lambda = \partial D^l_\lambda$ be the sphere of radius $\lambda$. Set
        \begin{equation}
            S_{0,b} = \left\{x \in M \times S^{l-1}_\lambda \,|\, -\nabla K_b \textrm{ is tangent to } M \times S^{l-1}_\lambda \textrm{ at } x\right\}
        \end{equation}
        By non-degeneracy of the difference functions (the second condition in Definition~\ref{def:2}), this is a transversally cut out submanifold of $M \times S^{l-1}_\lambda$ for each $b$, together forming a smooth subbundle of the trivial bundle $\{M \times S^{l-1}_\lambda\}_b$ over $B$. We choose a smooth family of small tubular neighbourhoods $\{U_b\}_b$ of $\{S_{0,b}\}_b$ in $M \times D^l_\lambda$, such that $-\nabla K_b$ is tangent to $\partial U_b \setminus M \times S^{l-1}_\lambda$; see Figure \ref{fig:1}. 
        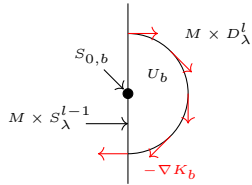
\begin{figure}[h]  
            \centering
            \begin{tikzpicture}[scale=.8,
            v/.style={draw,shape=circle, fill=black, minimum size=1.3mm, inner sep=0pt, outer sep=0pt},
            vred/.style={draw,shape=circle, fill=red, minimum size=1mm, inner sep=0pt, outer sep=0pt},
            vsmall/.style={draw,shape=circle, fill=black, minimum size=1mm, inner sep=0pt, outer sep=0pt}]
                \draw (0,1.5) to (0,-1.5);
                \node[v] at (0,0) {};
                \draw (0,-1) arc (-90:90:1);
                \draw[color=red, ->] (0, 1) to (0.5,1);
                \draw[color=red, ->] (1, 0) to (1, -0.5);
                \draw[color=red, ->] (0, -1) to (-0.5, -1);
                \draw[color=red, ->] (0.707, 0.707) to (1.061,0.354);
                \draw[color=red, ->] (0.707, -0.707) to (0.354,-1.061);

                \draw[->] (-0.4, 0.4) to (-0.07,0.07);
                \node at (-0.6, 0.6) {\tiny $S_{0,b}$};

                \node at (0.5, 0.3) {\tiny $U_b$};

                \draw[->] (-0.7, -0.5) to (0, -0.5);
                \node at (-1.3, -0.4) {\tiny $M \times S_\lambda^{l-1}$};
                
                \node at (1.5,1) {\tiny $M \times D^l_\lambda$};

                \node at (0.7, -1.3) {\tiny ${\color{red} -\nabla K_b}$};
            \end{tikzpicture}
            \caption{Local model near $S_{0,b}$.}
            \label{fig:1}
        \end{figure}

        We let $A_{M,b} = (M \times D^{l}_\lambda) \setminus U_b^\circ$. The family $\bA_M = \{A_{M,b}\}_b$ forms a smooth fibre bundle over $B$; note the fibres are naturally manifolds with corners.
        We further set:
        \begin{equation}
            \partial_- A_{M,b} = \overline{
                \left\{ x \in \partial A_{M,b} \,|\, -\nabla K_b \textrm{ points outwards at } x\right\}
            }
        \end{equation}
        where $\overline \cdot$ denotes the closure, and
        \begin{equation}
            \partial_+ A_{M,b} = \overline{
                \left\{ x \in \partial A_{M,b} \,|\, -\nabla K_b \textrm{ points inwards at } x\right\}
            }
        \end{equation}
        On the rest of the boundary (so $\partial A_{M,b} \setminus \partial_\pm A_{M,b}$), $-\nabla K_b$ is tangent to the boundary; see Figure \ref{fig:2}. 

        \begin{figure}[h]
            \centering
            \begin{tikzpicture}[scale=.8,
            v/.style={draw,shape=circle, fill=black, minimum size=1.3mm, inner sep=0pt, outer sep=0pt},
            vred/.style={draw,shape=circle, fill=red, minimum size=1mm, inner sep=0pt, outer sep=0pt},
            vsmall/.style={draw,shape=circle, fill=black, minimum size=1mm, inner sep=0pt, outer sep=0pt}]
                \draw (0, 0) rectangle (2, 2);
                
                \draw[->, red] (0,2) to (0,1.5);
                \draw[->, red] (1,2) to (1,1.5);
                \draw[->, red] (2,2) to (2,1.5);

                \draw[->, red] (0,1) to (0,0.5);
                \draw[->, red] (2,1) to (2,0.5);

                \draw[->, red] (0,0) to (0,-0.5);
                \draw[->, red] (1,0) to (1,-0.5);
                \draw[->, red] (2,0) to (2,-0.5);

                \node at (1, 2.4) {\tiny $\partial_+ A_{M,b}$};
                \draw[->] (0.8, 2.3) to (0.8, 2);
                
                \node at (1, 0.25) {\tiny $\partial_- A_{M,b}$};
                \draw[->] (0.8, 0.2) to (0.8, 0);

                \node at (3.65, 1) {\tiny $\partial A_{M,b} \setminus \partial_\pm A_{M,b}$};
                \draw[->] (2.3,1) to (2,1);

                \node at (2.5, -0.5) {\tiny ${\color{red} -\nabla K_b}$};

                \node at (1, 1) {\tiny $A_{M,b}$};
            \end{tikzpicture}
            \caption{Flow of $-\nabla K_b$ near the boundary of $A_{M,b}$.}
            \label{fig:2}
        \end{figure}
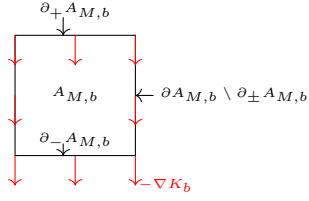

        Let $\pi_{M,b}: \partial_- A_{M,b} \to M$ be the projection.
        \begin{LL}\label{lem: piMb fib}
            $\{\pi_{M,b}\}_b$ is a smooth family of smooth fibre bundles, which has fibres diffeomorphic to $S^{2k-1} \times D^{l-2k}$.

        \end{LL}
        \begin{proof}
            Let $x \in M$. Each $K_{b,x} := K_b|_{\{x \} \times \bR^l}: \bR^l\to \bR$ is smoothly homotopic through functions of tube type to a nondegenerate quadratic form, by definition. The quadratic form must have signature $(l-2k, 2k)$, so each $K_{b,x}$ must be smoothly homotopic (through functions of tube type) to the function

            \begin{equation*}
                (x_1, \ldots, x_l) \mapsto (x_1^2 + \ldots + x_{l-2k}^2) - (x_{l-2k+1}^2 + \ldots + x_l^2)
            \end{equation*}
            It follows that each fibre of $\pi_{M, b}$ must be smooth and diffeomorphic to $S^{2k-1} \times D^{l-2k}$.

            The same argument applies over small discs in $M$, showing that each $\pi_{M,b}$ is a locally trivial smooth fibre bundle.
        \end{proof}
        Let $\bA_L = \{A_{L,b}\}_b$ be a smooth family of tubular neighbourhoods of each $L_b$ in $M \times D^l_\lambda$. Near each $L_b$, we can choose $A_{L,b}$ to be given by a similar local model to $A_{M,b}$, with codimension 0 submanifolds $\partial_\pm A_{L,b}$ of the boundary on which $-\nabla K_b$ points inwards/outwards respectively. Since each $K_b$ has no critical points outside $A_{L,b}$, by replacing $A_{L,b}$ with its image under the flow of $-\nabla K_b$ we may assume that $\partial_- A_{L,b} \subseteq \partial_- A_{M,b}$, as shown in Figure \ref{fig:3}. This whole procedure can be done smoothly in families, and so we obtain two smooth fibre bundles $\{\partial_- A_{L,b}\}_b$ and $\{\partial_- A_{M,b}\}_b$, with a smooth family of embeddings $\{j_b: \partial_- A_{L,b} \hookrightarrow \partial_- A_{M,b}\}_b$ given by the natural inclusions. We write $\{\pi_{L,b}: \partial_- A_{L,b} \to L_b\}_b$ for the natural projection maps. These are all smooth fibre bundles by the same proof as Lemma \ref{lem: piMb fib}.

        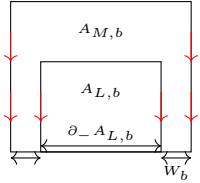
\begin{figure}[h]
            \centering
            \begin{tikzpicture}[scale=.8,
            v/.style={draw,shape=circle, fill=black, minimum size=1.3mm, inner sep=0pt, outer sep=0pt},
            vred/.style={draw,shape=circle, fill=red, minimum size=1mm, inner sep=0pt, outer sep=0pt},
            vsmall/.style={draw,shape=circle, fill=black, minimum size=1mm, inner sep=0pt, outer sep=0pt}]
                \draw (1, 0) rectangle (3, 1.5);
                \draw (0.5, 0) rectangle (3.5, 2.5);

                \node at (2, 1) {\tiny $A_{L,b}$};
                \node at (2, 2) {\tiny $A_{M,b}$};

                \draw[<->] (1, 0.1) to (3, 0.1);
                \draw[<->] (3, -0.1) to (3.5, -0.1);
                \draw[<->] (0.5, -0.1) to (1, -0.1);

                \node at (2, 0.3) {\tiny $\partial_- A_{L,b}$};
                \node at (3.25, -0.35) {\tiny $W_b$};

                \draw[red, ->] (0.5, 2) to (0.5, 1.5);
                \draw[red, ->] (0.5, 1) to (0.5, 0.5);

                \draw[red, ->] (3.5, 2) to (3.5, 1.5);
                \draw[red, ->] (3.5, 1) to (3.5, 0.5);

                \draw[red, ->] (1, 1) to (1, 0.5);
                \draw[red, ->] (3, 1) to (3, 0.5);
            \end{tikzpicture}
            \caption{$A_{L,b}$ inside $A_{M,b}$.}
            \label{fig:3}
        \end{figure}

        We define $\bW = \{W_b\}_b$ to be the closure of the fibrewise complement: 
        \begin{equation*}
            W_b = \overline{
                \partial_- A_{M,b} \setminus j_b(\partial_- A_{L,b})
            }.
        \end{equation*}
        Each $\partial W_b$ has two components, $\partial (\partial_- A_{M,b})$ and $\partial (\partial_- A_{L,b})$.
        \begin{LL}
            Each $W_b$ is an $h$-cobordism.
        \end{LL}
        \begin{proof}
            $\partial_-A_{M,b}$ and $\partial_-A_{L,b}$ are fibrations over $M$ and $L$, both with homotopy fibre $S^{2k-1}$. The inclusion $\partial_- A_{M,b} \to \partial_+ A_{L,b}$ covers the homotopy equivalence $L_b \to M$. By the long exact sequence of a fibration, we may conclude that the maps between $\partial_- A_{M,b}$, $\partial_- A_{L,b}$, $L_b$ and $M$ all induce isomorphisms on $\pi_1$. 

            Using the upwards gradient flow of $K_b$, we obtain a diffeomorphism $A_{M,b} \cong (A_{L,b} \cup W_b \times I)$. Since the inclusion $A_{L,b} \to A_{M,b}$ is a homotopy equivalence, there is therefore a deformation retraction from $W_b$ to $\partial(\partial_- A_{L,b})$, and so the inclusion $\partial(\partial_- A_{L,b}) \to W_b$ is a homotopy equivalence..

            Since $A_{M,b}$ and $A_{L,b}$ are obtained from $\partial_- A_{M,b}$ and $\partial_- A_{L,b}$ respectively via attaching handles of homotopical dimension $\geq 2k$, the inclusion $W_b \to A_{M,b}$ also induces an isomorphism on $\pi_1$.

            By Alexander duality, the inclusion of the other boundary component $\partial(\partial_- A_{M,b}) \to W_b$ induces an isomorphism on homology. This also works with $\bZ[\pi_1]$-coefficients, and so we deduce that this inclusion is also a homotopy equivalence.
        \end{proof}
        By construction, there is a smooth family of diffeomorphisms
        \begin{equation}
            \left\{
            \partial_- A_{M,b} \cong \partial_- A_{L,b} \cup_{\partial (\partial_- A_{L,b})} W_b
            \right\}_b
        \end{equation}
        \begin{LL}\label{lem: disc bun L}
            There is a smooth family of maps $\{p_{L,b}: \partial_- A_{L,b} \to L_b \times S^{2k-1}\}_b$, such that each $p_{L,b}$ is a fibre bundle with fibres diffeomorphic to $D^l$ and structure group $O(l)$.
        \end{LL}
        We may therefore choose a smooth family of smooth sections $\{s_{L,b}: L_b \times S^{2k-1} \to \partial_- A_{L,b}\}_b$ of $\{p_{L,b}\}_b$, not touching the boundaries $\partial (\partial_- A_{L,b})$. 
        \begin{proof}
            Follows from the fact that each $\partial_- A_{L,b}$ is a tubular neighbourhood of the boundary of the intersection of $\partial A_{L,b}$ with the negative eigenbundle of the Hessian of $K_b$, and that this vector bundle is trivial by Theorem~\ref{prop:1}(\ref{prop1: pt 2}). 
        \end{proof}
    \subsection{Constructing the map $\delta$}

        Let $g'_b = g_b \times \operatorname{Id}: L_b \times S^{2k-1} \to M \times S^{2k-1}$; together these form a smooth family of smooth maps. Let $p'_M: M \times S^{2k-1} \to M$ be the projection to the first factor.

        \begin{PP}
            Assume $l \gg 0$. Then there is a smooth family of smooth embeddings $\{s_{M,b}: M \times S^{2k-1} \to \partial_- A_{M,b}\}_b$ such that for each $b$, the following diagram commutes:
            \begin{equation}\label{eq: sMb,piMb,p'M}
                \begin{tikzcd}
                    M \times S^{2k-1} 
                    \arrow[rr, "s_{M,b}"]
                    \arrow[dr, "p'_M"]
                    &&
                    \partial_- A_{M,b}
                    \arrow[dl, "\pi_{M,b}"]
                    \\
                    & M &
                \end{tikzcd}
            \end{equation}
            
            We may additionally assume that the image of each $s_{M,b}$ is disjoint from each $\partial (\partial_- A_{M,b})$.
        \end{PP}

        \begin{proof}
            Since each $g_b$ is a homotopy equivalence, so is each $g'_b$, and furthermore there exists a smooth (in $B$) family of homotopy inverses. Choosing such a family and composing with $s_{L,b}$ and $j_b$, we obtain a smooth family of smooth maps $\{s_{M,b}: M \times S^{2k-1} \to \partial_- A_{M,b}\}_b$, such that the following diagram commutes up to homotopy (such that the homotopies may be chosen to vary smoothly in $b$):

            \begin{equation}\label{eq:6}
                \begin{tikzcd}
                    \partial_- A_{L,b}
                    \arrow[r, "j_b"]
                    &
                    \partial_- A_{M,b}
                    \\
                    L_b \times S^{2k-1}
                    \arrow[u, "s_{L,b}"]
                    \arrow[r, "g'_b"]
                    &
                    M \times S^{2k-1}
                    \arrow[u, dashed, "s_{M,b}"]
                \end{tikzcd}
            \end{equation}
            Note that all maps in this diagram are homotopy equivalences. 

            In particular, (\ref{eq: sMb,piMb,p'M}) commutes up to a (smoothly varying in $B$) family of homotopies.

            Since each $\pi_{M,b}: \partial_- A_{M,b} \to M$ is a smooth fibre bundle, by homotoping $\{s_{M,b}\}_b$ (as a family of maps over $B$) we may assume that (\ref{eq: sMb,piMb,p'M}) commutes strictly.

            Perturbing $\{s_{M,b}\}_b$ if necessary, we may assume its image is disjoint from all $\partial (\partial_- A_{M,b})$. 

            $s_{M,b}$ is a map of codimension $l$, so for $l \gg 0$, after a generic perturbation it becomes an embedding. Furthermore, for $l$ large enough, this can be done fibrewise over $M$ so we can perform this perturbation while preserving commutativity of (\ref{eq: sMb,piMb,p'M}).
        \end{proof}

        Let $\nu_{s_{M,b}} \to M \times S^{2k-1}$ be the normal bundle of each $s_{M,b}$. These assemble to form a smooth family of vector bundles $\{\nu_{s_{M,b}} \to M \times S^{2k-1}\}_b$.
        \begin{LL}
            For $l \gg 0$, the family of vector bundles $\{\nu_{s_{M,b}}\}_b$ is trivial. 
        \end{LL}

        \begin{proof}
            By construction of $A_{M,b}$, there is a smooth family of codimension 0 embeddings $\{A_{M,b} \to M \times D^l\}_b$. Therefore there is a smooth family of isomorphisms $\{TA_{M,b} \cong TM \oplus \bR^l\}_b$.

            $\partial_- A_{M,b}$ is a codimension 0 submanifold of $\partial A_{M,b}$, so there is a smooth family of isomorphisms $\{T\partial_- A_{M,b} \oplus \bR \cong TM \oplus \bR^l\}_b$. Similarly, there is an isomorphism $T(M \times S^{2k-1}) \oplus \bR \cong TM \oplus \bR^{2k}$.

            It follows that the family of vector bundles $\{\nu_{s_{M,b}}\}_b$ is trivial after stabilising; if $l$ is sufficiently large, we do not need to stabilise.
        \end{proof}

        \begin{CC}\label{cor: tub nbds}
            There is a smooth family of tubular neighbourhoods 
            \begin{equation}
                \left\{u_b: M \times S^{2k-1} \times D^{l-2k} \hookrightarrow \partial_- A_{M,b}\right\}_b
            \end{equation}
            with images disjoint from $\partial(\partial_- A_{M,b})$, and such that the restriction of $u_b$ to $M \times S^{2k-1} \times \{0\}$ agrees with $s_{M,b}$. 
            
            These may be chosen to commute with the projection maps to $M$.
        \end{CC}

        We define $Z_b$ to be the complement of each such tubular neighbourhood:
        \begin{equation}
            Z_b = \partial_- A_{M,b} \setminus \operatorname{Im}(u)^\circ
        \end{equation}
        These form a smooth fibre bundle $\{Z_b\}_b$ over $B$, with a smooth family of inclusions of boundary components $\{u_b: M \times S^{2k-1} \times S^{l-2k-1} \hookrightarrow Z_b\}_b$. By Lemma \ref{lem:codim3hcob} each $Z_b$ is an $h$-cobordism on $M \times S^{2k-1} \times S^{l-2k-1}$.
        \begin{LL}
            Each $\pi_{M,b}|_{Z_b}: Z_b \to M$ is a smooth fibre bundle.
        \end{LL}
        \begin{proof}
            Each $\pi_{M,b}: \partial_- A_{M,b} \to M$ is a smooth fibre bundle, and by Corollary \ref{cor: tub nbds}, each $\operatorname{Im}(u_b)$ is a smooth subbundle.
        \end{proof}

        We now choose a fixed point $x_0 \in M$. We set $Y_b = (\pi_{M,b}|_{Z_b})^{-1}\{x_0\}$ to be the fibre of $Z_b$ over $x_0$; these together form a smooth fibre bundle $\{Y_b\}_b$ over $B$ with each fibre an $h$-cobordism on $S^{2k-1} \times S^{l-2k-1}$. By Proposition \ref{prop: tot hcob param}, there is a smooth family of diffeomorphisms 
        \begin{equation}\label{eq: erwdsghdrgg}
            \{Z_b \cong M \times Y_b\}_b
        \end{equation}
        relative to $M \times S^{2k-1} \times S^{l-2k-1}$ (note that these diffeomorphisms need not be fibred over $M$ though).
        Choose a small (codimension 0) embedding $u: D^{l-2} \hookrightarrow S^{2k-1} \times S^{l-2k-1}$. By Corollary \ref{cor: conn est}, for $l \gg 0$ large enough there is a smooth family of $h$-cobordisms on $D^{l-2}$, which we call $\{V_b\}_b$ (in particular, this includes the data of a smooth family of embeddings into the boundary $\{D^{l-2} \hookrightarrow V_b\}_b$), along with a smooth family of diffeomorphisms:
        \begin{equation}
            \left\{Y_b \cong \left(S^{2k-1} \times S^{l-2k-1} \times [0,1] \right) \cup_{\operatorname{Im}(u) \times \{1\}} V_b \right\}_b
        \end{equation}
        Using (\ref{eq: erwdsghdrgg}) and by gluing $M \times S^{2k-1} \times D^{l-2k}$ back in along its boundary, we obtain:
        \begin{CC}\label{cor: phi}
            There is a smooth family of diffeomorphisms $\{\phi_b\}_b$:
            \begin{equation}
                \left\{\phi_b: \partial_- A_{M,b} \to M \times \left( (S^{2k-1} \times D^{l-2k}) \cup_{\operatorname{Im}(u)} V_b\right)\right\}_b
            \end{equation}
            such that for each $b$, $\phi_b \circ s_{M,b}$ is the inclusion $M \times S^{2k-1} \hookrightarrow M \times \left((S^{2k-1} \times D^{l-2k}) \cup V_b\right)$ of the zero-coordinate of $D^{l-2k}$: in particular, the image of $\phi_b \circ s_{M,b}$ does not touch $V_b$.
            
        \end{CC}

    \subsection{Reduction to a product $h$-cobordism}
        Choose a smooth family of retractions $\{r_b: V_b \to D^{l-2}\}_b$. Let $p_{12}: M \times S^{2k-1} \times D^{l-2k} \to M \times S^{2k-1}$ be the projection map to the first two factors, let $q_b = p_{12} \circ r_b: V_b \to M \times S^{2k-1}$, and set
        \begin{equation}
            p_{M,b} = q_b \circ \phi_b: \partial_- A_{M,b} \to M \times S^{2k-1}
        \end{equation}
        Then $\{p_{M,b}\}_b$ is a smooth family of maps, but each $p_{M,b}$ is not necessarily a fibre bundle.
        \begin{RR}
            Each $s_{M,b}$ is a section of $p_{M,b}$:
            \begin{align*}
                p_{M,b} \circ s_{M,b} 
                &= q_b \circ \phi_b \circ s_{M,b}
                \\
                &= \operatorname{Id}_{M \times S^{2k-1}}
            \end{align*}
            The first equality is by definition and the second is by Corollary \ref{cor: phi}.
        \end{RR}
        We define 
        \begin{equation}
            \iota_b = \phi_b \circ j_b \circ s_{L,b}: L_b \times S^{2k-1} \hookrightarrow M \times \left( (S^{2k-1} \times D^{l-2k}) \cup_{\operatorname{Im}(u)} V_b\right)
        \end{equation}
        These form a smooth family of embeddings $\{\iota_b\}_b$.

        \begin{LL}
            Assume $l \gg 0$. Then $\{\iota_b\}_b$ is isotopic (through smooth families of embeddings) to a smooth family of embeddings $\{\psi_b\}_b$ of the form
            \begin{equation}
                \psi_b: (x, s) \mapsto (g_b(x), s, h_b(x))
            \end{equation}
            for some smooth family of embeddings $\{\theta_b = (g_b, h_b): L_b \hookrightarrow M \times (D^{l-2k})^\circ \}_b$.
        \end{LL}
        Note that each $\psi_b$ does not hit $V_b$.
        \begin{proof}
            By homotopy commutativity of (\ref{eq:6}), there is a smooth family of homotopies $\{s_{M,b} \circ g'_b\}_b \simeq \{j_b \circ s_{L,b}\}_b$. Composing with $p_{M,b}$, we find that $\{g'_b\}_b \simeq \{q_b \circ \iota_b\}_b$. 

            Since $g'_b = g_b \times \operatorname{Id}_{S^{2k-1}}$, and $q_b$ is a homotopy equivalence, this implies that the family of embeddings $\{\iota_b\}_b$ is homotopic to a family of maps (but not necessarily embeddings) $\{\psi_b\}_b$ of the desired form (one may take any choice of $h_b: L_b \to (D^{l-2k})^\circ$, since the disc is contractible). By assuming $l \gg 0$ is sufficiently large, we can generically perturb both $\{\theta_b\}_b$ and the homotopy so that they are embeddings, as required.
        \end{proof}

        By Lemma \ref{lem: disc bun L}, $\{\phi_b \circ j_b(\partial_- A_{L,b})\}_b$ is a smooth family of tubular neighbourhoods of $\{\operatorname{Im}(\iota_b)\}_b$. By isotopy extension, there is a smooth family of embeddings $\{j'_b: \partial_- A_{L,b} \hookrightarrow \partial_- A_{M,b}\}_b$, such that $\{\phi_b \circ j'_b(\partial_- A_{L,b})\}_b$ form a smooth family of tubular neighbouhoords of $\{\operatorname{Im}(\psi_b)\}_b$. Since isotopic submanifolds have diffeomorphic complements, we find:
        \begin{CC}\label{cor: Wb simplify}
            There is a smooth family of diffeomorphisms
            \begin{equation}
                \left\{ W_b \cong \left( M \times \left( (S^{2k-1} \times D^{l-2k}) \cup_{\operatorname{Im}(u)} V_b \right) \right) \setminus \operatorname{Im}(\phi_b \circ j'_b)\right\}_b
            \end{equation}
            restricting to the natural diffeomorphism $\phi_b \circ j'_b \circ j_b^{-1}$ between the boundary components $\{j_b(\partial(\partial_- A_{L,b}))\}_b$ and $\{\phi_b \circ j'_b(\partial(\partial_- A_{L,b}))\}_b$
        \end{CC}

        Since $\psi_b = \theta_b \times Id_{S^{2k-1}}$ splits as a product, this admits a smooth family of tubular neighbourhoods $\{T_b \times S^{2k-1} \}_b$, where $\{T_b \subseteq M \times (D^{l-2k})^\circ\}_b$ is a smooth family of tubular neighbourhoods of the embeddings $\{\theta_b\}_b$. By Lemma \ref{lem:codim3hcob}, the complement of each $T_b$ in $M \times D^{l-2k}$ is an $h$-cobordism.
        By uniqueness of tubular neighbourhoods and Corollary \ref{cor: Wb simplify}, we find:
        \begin{CC}\label{cor: hcob is prod}
            There is a smooth family of diffeomorphisms
            \begin{equation}
                \left\{ W_b \cong \left( (M \times D^{l-2k}) \setminus T_b \right) \times S^{2k-1} \cup_{M \times \operatorname{Im}(u)} M \times V_b\right\}_b
            \end{equation}
        \end{CC}
        \begin{proof}[Proof of Theorem \ref{thm:mainoverB}]
            By Corollary \ref{cor: hcob is prod} and Theorem \ref{thm: prod decomp} (and the fact $\chi(S^{2k-1}) = 0$), we obtain a smooth family of diffeomorphisms
            \begin{equation}
                \left\{W_b \cong \partial (\partial_- A_{L,b}) \times [0,1] \cup_{\operatorname{Im}(v)} M \times V_b\right\}
            \end{equation}
            where $v_b: M \times D^{l-2} \hookrightarrow \partial(\partial_- A_{L,b})$ is given by composing $\Id_M \times u$ with the corresponding diffeomorphism. The embeddings $\{v_b\}$ are such that the composition of $v_b$ with $g_b \circ \pi_{L,b}$ is homotopic to the projection $M \times D^{l-2} \to M$, smoothly in $b$. 
            
            Gluing back in $\{\partial_- A_{L,b}\}_b$, we obtain a smooth family of diffeomorphisms 
            \begin{equation}\label{eq: aaa}
                \left\{\partial_- A_{M,b} \cong \partial_- A_{L,b} \cup_{\operatorname{Im}(v)} M \times V_b\right\}_b
            \end{equation}
            such that the induced embedding $\partial_- A_{L,b} \to \partial_- A_{M,b}$ is isotopic to the family $j_b$.

            Let $\{V'_b\}_b$ be a smooth family of inverses for $\{V_b\}_b$.
            Let $\{v'_b: M \times D^{l-2} \hookrightarrow \partial(\partial_- A_{L,b})\}_b$ be a smooth family of embeddings which is a small perturbation of $\{v_b\}_b$ off itself. Then gluing $\{V'_b\}_b$ into (\ref{eq: aaa}) along $\{\operatorname{Im}(v'_b)\}$ gives a family of diffeomorphisms
            \begin{equation}
                \partial_- A_{M,b} \cup_{\operatorname{Im}(w_b)} (M \times V'_b) \cong \partial_- A_{L,b}
            \end{equation}
            isotopic to the embeddings $j_b:\partial_- A_{L,b} \to \partial_- A_{M,b}$, 
            where $w_b: M \times D^{l-2} \hookrightarrow \partial_- A_{M,b}$ is the composition of $v'_b$ with the inverse of \ref{eq: aaa}.
Recall $A_{M,b}$ is endowed with a function $K_b$ which has no critical point outside of $A_{L,b}$. In particular its vertical boundary $\del^v A_{M,b}$
is trivialised as $\del(\del_- A_{M,b})\times I$ using the flow of $\nabla K_b$. We can thus glue $V'_b\times I$ along the vertical boundary of $A_{M,b}$
to obtain another disk bundle $A'_{M,b}$:
\[A'_{M,b}=A_{M,b}\cup_{\del^v A_{M,b}} V_b'\times I.\]
As we have shown above, the family of embeddings $j_b:\del_- A_{L,b} \to \del_- A'_{M,b}$ induced by the flow of $-\nabla K_b$ is isotopic to a diffeomorphism and it follows that $A_{L,b}\to A'_{M,b}$ is also isotopic to a diffeomorphism. Indeed we construct this isotopy in two steps.  First saturate $A_{L,b}$ under the flow of $\nabla K_b$ by an isotopy so that after the isotopy $\del_+ A_{L,b}\subset \del_+ A_{M,b}$ and $\del_- A_{L,b}\subset \del_- A_{M,b}$. Second consider the product of the isotopy $\del_- A_{L,b} \to \del_- A_{M,b}$ by the vertical $I$-factor obtained by flowing along $\nabla K_b$.

            The composition $\{p_{M,b} \circ w_b\}_b$ is homotopic to the projection $M\times D^{l-2} \to M$, so for $l \gg 0$, $\{w_b\}_b$ is isotopic to the embedding $\operatorname{Id}_M \times i$, where $i$ is as in the statement of Theorem \ref{thm:mainoverB}.
So we obtain the required family of diffeomorphims with $E_b=A_{L,b}$ (after smoothing corners). Finally we observe that the diagram \eqref{eq:mainoverB} homotopy commutes by construction.
        \end{proof}

\section{Mapping class groups of tori}\label{sec: mcg}

    In this section, we use Hatcher and Hsiang-Sharpe's computation of the mapping class groups of high-dimensional tori to prove Corollary \ref{cor: mcg tori}.
\subsection{Background on automorphism groups}  
    We begin by quickly recapping some notions from the study of high-dimensional geometric topology. We refer to \cite[Section 1]{WW:I} for a more comprehensive introduction to the subject, and to \cite[Section 2.2]{HLLOR} for more a precise definitions of block automorphism spaces. Let $M$ be a closed manifold.
\subsubsection*{Important spaces}
    \begin{DD}\label{def:spac}
        \begin{itemize}
            \item We write $G(M)$ for the space of homotopy autoequivalences of $M$.
            \item We write $\Diff(M)$ for the space of self-diffeomorphisms of $M$. 
            \item We write $\Top(M)$ for the space of self-homeomorphisms of $M$.
            \item We write $\widetilde{\Diff}(M)$ for the \emph{block diffeomorphism group} of $M$. Roughly, this is (the realisation of) the semisimplicial group $\widetilde{\Diff}(M)_\bullet$ whose $k$-simplices $\widetilde{\Diff}(M)_k$ consist of the set of self-diffeomorphisms of $M \times \Delta^k$, which are required to respect the boundary stratification coming from that of $\Delta^k$, and behave nicely in collar neighbourhoods . 

            This should be compared to a similar model for $\Diff(M)$, given by the realisation of a semisimplicial group defined the same except restricting to diffeomorphisms which are compatible with the projection map to $\Delta^k$.
            \item We write $\widetilde{\Top}(M)$ for the \emph{block homeomorphism group} of $M$, defined analogously to $\widetilde{\Diff}(M)$.
            \item The \emph{smooth structure space} of $M$ is $\cS(M) := G(M) / \Diff(M)$.
            \item The \emph{topological structure space} of $M$ is $\cS^{\Top}(M) := G(M) / \Top(M)$.
        \end{itemize}
        The groups $G(M)$, $\Diff(M)$ and $\Top(M)$ we will often treat as simplicial groups, whose $k$-simplices are automorphisms (in the appropriate category) of $\Delta^k \times M$ which are compatible to the projection to $\Delta^k$.
    \end{DD}
    The \emph{parametrised Whitehead torsion} can be constructed as a map $\cS(M) \to \cH(M)$, see Remark~\ref{rem:spacelevelw} or \cite[Section 1.5]{WW:Survey}.

    Note that $\Diff(M)$ includes into $\widetilde{\Diff}(M)$ in a natural way (and $\Top(M)$ into $\widetilde{\Top}(M)$ similarly). Essentially by definition, there are homotopy fibration sequences:
    \begin{equation}\label{eq: fib 1}
        \Diff(T^n) \to G(T^n) \to \cS(T^n)
    \end{equation}

    \begin{equation}\label{eq: fib 2}
        \widetilde{\Diff}(M) / \Diff(M) \to G(M)/\Diff(M) = \cS(M) \to G(M)/\widetilde{\Diff}(M)
    \end{equation}

    \begin{equation}\label{eq: fib 2 TOP}
        \widetilde{\Top}(M) / \Top(M) \to G(M)/\Top(M) = \cS^\Top(M) \to G(M)/\widetilde{\Top}(M)
    \end{equation}
    The map including diffeomorphisms into homeomorphisms induces a map from (\ref{eq: fib 2}) to (\ref{eq: fib 2 TOP}). We write $c$ for the first map in (\ref{eq: fib 2}). The second map in (\ref{eq: fib 2}) uses the identification $\widetilde G(M) \simeq G(M)$, where $\widetilde G(M)$ is the semisimplicial group of block homotopy autoequivalences of $M$.

    \begin{RR}
        It is convenient to work with the following model for $\widetilde{\Diff}(M)/\Diff(M)$, and similarly for the other quotients of semisimplicial groups we encounter. We define a classifying space $B\widetilde{\Diff}(M)$ using a bar construction (see \cite[Section 2]{Ebert-Randal-Williams} for a more precise definition). We then define the quotient $\widetilde{\Diff}(M)/\Diff(M)$ to be the homotopy fibre of the induced map $B\widetilde{\Diff}(M) \to B\Diff(M)$.
    \end{RR}

    \begin{RR}
        The space $\widetilde{\Diff}(M) / \Diff(M)$ can be analysed (in a range) using pseudoisotopy theory, following Weiss-Williams \cite{WW:I} and Hatcher \cite{Hat78}. The space $G(M)/\widetilde{\Diff}(M)$ can be analysed using surgery theory. Together, this comprises the surgery-pseudoisotopy approach to understanding the homotopy type of diffeomorphism groups.
    \end{RR}

    \begin{DD}
        We define the \emph{topological $h$-cobordism space of $M$}, $H^\Top(M)$, exactly the same as $H(M)$ in Section \ref{sec: hcob}, except we allow topological rather than smooth $h$-cobordisms. 

        The stabilisation map $H^\Top(M) \to H^\Top(M \times I)$ is defined similarly to the smooth case, and the \emph{stable} topological $h$-cobordism space is defined to be the colimit over $k$ of $H^\Top(M \times I^k)$. 

        The \emph{topological parametrised Whitehead torsion} $w^\Top: \cS^\Top(M) \to \cH^\Top(M)$ is defined identically to the smooth setting.
    \end{DD}
    From the definitions, we have:
    \begin{LL}
        The following diagram commutes:
        \begin{equation}
            \begin{tikzcd}
                \cS(M) 
                \arrow[r]
                \arrow[d, "w"]
                &
                \cS^\Top(M)
                \arrow[d, "w^\Top"]
                \\
                \cH(M) 
                \arrow[r]
                &
                \cH^\Top(M)
            \end{tikzcd}
        \end{equation}
    \end{LL}

\subsubsection*{Pseudoisotopy theory}
    \begin{DD}\label{DD:pseu}
        A \emph{pseudoisotopy of $M$} is a diffeomorphism $F: M \times I \to M \times I$ which is required to be the identity over $M \times \{0\}$ (and also over $\partial M \times I$ if $M$ has boundary)\footnote{This should be compared to an \emph{isotopy}, which is defined similarly, but is required to respect the projection maps to $I$ on each side.}. We write $P(M)$ for the space of pseudoisotopies on $M$; this is a group under composition.

        Similarly to the case of $h$-cobordisms, there is a stabilisation map $\sigma: P(M) \to P(M \times I)$ sending $F$ to $F \times \Id_I$ (and then unbending corners appropriately, see \cite[Figure 1]{Hat78}). 
        
        We define the \emph{stable pseudoisotopy space} via
        \begin{equation}
            \cP(M) := \operatorname{colim}_k P(M \times I^k)
        \end{equation}
    \end{DD}
    \begin{RR}
        In the literature, pseudoisotopies are also known synonymously as \emph{concordances}.
    \end{RR}
    \begin{LL}
        There is a homotopy equivalence $P(M) \simeq \Omega H(M)$.

        This is compatible with stabilisation, and in particular there is a homotopy equivalence $\cP(M) \simeq \Omega \cH(M)$.
    \end{LL}
    \begin{proof}[Sketch proof]
        A based loop of $h$-cobordisms gives a fibre bundle $E$ over $S^1$, with fibre over $1$ identified with $M \times I$, along with an inclusion of fibre bundles $M \times I \to E$ into the boundary. Monodromy around $S^1$ then gives a diffeomorphism $M \times I \to M \times I$ which is a pseudoisotopy. Conversely, given a pseudoisotopy $F$ of $M$, its mapping torus is a loop of $h$-cobordisms on $M$. These constructions can be upgraded to give an inverse pair of homotopy equivalences.
    \end{proof}
    For $M$ of dimension at least 5, there is an involution $\tau$ on $h$-cobordism spaces $H(M)$, defined to send an $h$-cobordism $W$ on $M$ to (a choice of) classical inverse $W'$, viewed as an $h$-cobordism on $M$ (cf. \cite{WW:I}, \cite[Section 5]{Munoz-Echaniz}).
    \begin{RR}\label{RR:6654rter9098ygt}
        $\tau$ induces an involution on $\pi_1 H(M) = \pi_0 P(M)$; it will help to have an explicit description of this. Let $i: I \to I$ be the reflection $t \mapsto 1-t$.

        For a pseudoisotopy $F: M \times I \to M \times I$ from $g$ to the identity, $\tau(F)$ is obtained by inverting $F$, reflecting in the $I$-direction, and composing with an appropriate isotopy so that $F$ is still the identity over $0 \in I$. 
        
        More explicitly, $\tau(F) = (g \times i) \circ F^{-1} \circ (\Id_M \times i)$.
    \end{RR}
    \begin{LL}[{\cite[Appendix I]{Hat78}, \cite[Lemma B.1]{Munoz-Echaniz}}]\label{lem: stab sign}
        Up to homotopy, the involution $\tau$ anticommutes with the stabilisation map $\sigma$; explicitly, the map:
        \begin{equation}
            \tau \circ \sigma + \sigma \circ \tau: H(M) \to H(M \times I)
        \end{equation}
        is nullhomotopic.
    \end{LL}
    \begin{DD}
        Let $B$ be a compact manifold. We define an involution $\tau^{st}: \cH_B(M) \to \cH_B(M)$ on the stable $h$-cobordism space to be given by $(-1)^k \cdot \tau: H_B(M \times I^k) \to H_B(M \times I^k)$ on each unstable $h$-cobordism space $H_B(M \times I^k)$ (for $k \geq 6$ so that $H_B(M \times I^k)$ admits a homotopy group structure).
    \end{DD}
    
    Lemma \ref{lem: stab sign} implies $\tau^{st}$ is well-defined on $\cH_B(M)$ (and in particular is the reason for the sign in the definition).

    \begin{RR}
        By incorporating a choice of homotopy from Lemma \ref{lem: stab sign}, $\tau^{st}$ can be extended to a space-level homotopy involution on $\cH(M)$ (cf. \cite[Section 5.2]{Munoz-Echaniz}). However this is not necessary for our purposes.
    \end{RR}

    Let $h: \pi_0 P(M) \to \pi_1 \widetilde{\Diff}(M)/\Diff(M)$ be defined as follows. Let $[F] \in \pi_0 P(M)$ be a class, and $F$ a representative. Then $F: M \times I \to M \times I$ determines a 1-simplex in $\widetilde{\Diff}(M)$ with boundary on $\Diff(M)$, and so we may define $h([F])$ to be the corresponding element of $\pi_1 \widetilde{\Diff}(M)/\Diff(M)$.
    \begin{RR}
        The map $h$ is exactly the edge homomorphism $E^1_{00} \to E^\infty_{00}$ in the Hatcher spectral sequence \cite[Section 3]{Hat78} for $\pi_{*+1}\widetilde{\Diff}(M)/\Diff(M)$.
    \end{RR}
    \begin{LL}
        $h$ is surjective.
    \end{LL}
    \begin{proof}
        
        It follows from \cite[Propositions 2.3 and 2.7]{Ebert-Randal-Williams} (applied to the standard simplicial circle) that any based loop in $\widetilde{\Diff}(M)/\Diff(M)$ is represented by some block diffeomorphism $F: M \times I \to M \times I$.
        
        Setting $f := (F_0: M \to M)$, we obtain a pseudoisotopy $F' := f^{-1} \circ F$, and $h([F'])$ represents the same class as $F$ in $\pi_1$.
    \end{proof}
    \begin{PP}\label{lem: w on P}
        The following diagram commutes.
        \begin{equation}
            \begin{tikzcd}
                \pi_0 P(M) 
                \arrow[rr, "y \mapsto y+\tau^{st}(y)"]
                \arrow[d, "h"]
                &&
                \pi_0 \cP(M)
                \arrow[d, "\cong"]
                \\
                \pi_1 \widetilde{\Diff}(M)/\Diff(M)
                \arrow[r, "c"]
                &
                \pi_1 \cS(M)
                \arrow[r, "w"]
                &
                \pi_1 \cH(M)
            \end{tikzcd}
        \end{equation}
    \end{PP}
    \begin{proof}
        Let $F \in P(M)$. Consider the map $F \times 0: M \times I \to M \times D^1 \times I$. This is homotopic to the embedding $Q: M \times I \to M \times D^1 \times I$ defined by:
        \begin{equation}
            Q(x, t) = (F(x, t)-t, t)
        \end{equation}
        $Q$ respects the projection maps to $I$ and is an embedding. The image of $Q$ splits $M \times D^1 \times I$ into two components, one component $U^+_t = (U^+_t)_{t \in [0,1]}$ containing $M \times \{1\} \times I$ and the other one $U^-_t$ containing $M \times \{-1\} \times I$. Both $U^+_t$ and $U^-_t$ fibre over the circle $S^1 := I/(0 \sim 1)$ and represent loops of trivial $h$-cobordisms from $M\times\{\pm 1\}$ to a (possibly non-trivial) bundle over $S^1$ with fiber $M$. The monodromy over $S^1$ of these bundles provide a pair of pseudoisotopies $\varphi^+$ and $\varphi^-$ on $M$.

Following the definition of the parametrised Whitehead torsion $w$, in order to compute $w\circ c\circ h([F])$ we need to glue (classical) inverses $V^-_t$, $V^+_t$ to $U^-_t$ and $U^+_t$ along $M\times \{-1\}$ and $M\times \{1\}$ respectively and stabilise by taking a product with $I$ to obtain a loop of $h$-cobordisms on $M\times D^1$.

Now by construction $U^-_t$ and $U^+_t$ are classical inverses to each other since $U^-_t\cup_{Q(M\times I)} U^+_t = M\times D^1 \times I$ so we may take $V^-_t=U^+_t$
and $V^+_t=U^-_t$. After taking product with $I$ (namely using the stabilisation $\sigma$) to complete the construction of $w$ we obtain
\begin{align*}
    w \circ c\circ h([F])
    &=[\sigma(\varphi^-)] + [\sigma(\varphi^+)]
\end{align*}
in $\pi_0 P(M\times D^1)$.

Next observe that if a loop of trivial $h$-cobordisms $U_t$ on $M$ has monodromy $\varphi \in P(M)$, then the classical inverse $V_t$ has monodromy $(g\times \Id)\circ \varphi^{-1}$ where $g \in \Diff(M)$ is the restriction of $\varphi$ to $M\times \{1\}$.
Hence, viewing $U_t^+$ as a cobordism on $M= M \times \{1\}$, because we identify $J:= [-1,0]$ with $I$ via the map $t \mapsto -t$, the monodromy is conjugated by the reflection $i(t)=1-t$ of $I$ and (similarly to Remark \ref{RR:6654rter9098ygt}) we get
\[\varphi^+=(\Id\times i)\circ(g\times \Id)\circ (\varphi^-)^{-1}\circ (\Id\times i)=\tau(\varphi^-).\]
So we obtain
\begin{align*}
    w\circ c\circ h([F])
    &=[\sigma(\varphi^-)] +[\sigma(\tau(\varphi^-))]\\
    &=[\varphi^-]+\tau^{st}[\varphi^-]
\end{align*}
in $\pi_0\cP(M)$.
Finally we compute explicitly $\varphi^-$ using the following trivialisation (as a bundle over $I$) $M \times J \times I \to U^-$:
        \begin{equation}
            (x, s, t) \mapsto \begin{cases}
                (x,s,t)
                & \textrm{ if }
                s \in [-1,-t]
                \\
                (F(x,s + t) - t, t)
                & \textrm{ if }
                s \in [-t, 0]
            \end{cases}
        \end{equation}
where $J=[-1,0]$. We obtain the monodromy $[\varphi^-]=[F]$ by setting $t=1$ (up to translating by $1$ to identify $J$ with $I$).
So we obtain the desired result:
\[w\circ c\circ h([F])=[F]+\tau^{st}([F]).\]
  \end{proof}

\subsection{Tori}
    We recap Hatcher's computation in more detail. This proceeds by first using the fibration sequence (\ref{eq: fib 1}), then studying $\cS(T^n)$ using the fibration sequence (\ref{eq: fib 2}). Up to extensions, this determines the mapping class group of $T^n$ in terms of three other groups, which are then computed individually.

    Throughout this section, we fix some $n \geq 6$; many of these computations break down in lower dimensions.

\subsubsection*{The individual pieces}
    Since $T^n$ is a $K(\pi,1)$, there is a homotopy equivalence of spaces:
    \begin{equation}
        G(T^n) \simeq T^n \times GL_n(\bZ)
    \end{equation}
    In particular, $\pi_0 G(T^n) \cong GL_n(\bZ)$ and $\pi_1 G(T^n) \cong \pi_1 T^n \cong \bZ^n$.

    From \cite[Section 3]{Siebenmann} (cf. also \cite[Section 4(5)]{Hat78}\footnote{Note the denominator of the LHS in \textit{loc. cit.} is missing a $\widetilde \cdot$.}):
    \begin{equation}\label{eq: piece 1}
        \pi_1 G(T^n)/\widetilde{\Diff}(T^n) \cong \bigoplus_{i=0}^n \left(\pi_{i+1}\Top/O\right)^{\oplus \binom{n}{i}}
    \end{equation}
    Here $\Top/O$ is the colimit over $k$ of $\Top(\bR^k)/O(k)$. This is a finite group, by \cite[Essay V, Theorem 5.5(II)]{Kirby-Siebenmann}.

    Using the computations $\pi_2 T^n \cong 0$, $Wh_2(\pi_1 T^n) \cong 0$, that $\pi_1 T^n$ is abelian and that the first $k$-invariant of $T^n$ vanishes, \cite[Theorem 3.1]{Hat78} shows that there is an isomorphism of groups:
    \begin{equation}\label{eq: piece 2.1}
        \pi_1 \cH(T^n) \cong (\bZ/2)[t_1, t_1^{-1}, \ldots, t_n, t_n^{-1}]/(\bZ/2)[1]
    \end{equation}
    The same computation applies in the topological setting, and we obtain:
    \begin{LL}[{\cite[Theorem 3.1]{Hat78}}]\label{lem: Diff Top iso}
        The map $\pi_1 \cH(T^n) \to \pi_1 \cH^{Top}(T^n)$ is an isomorphism.
    \end{LL}

    \cite[Lemma 4.3]{Hatcher-Wagoner} (see also \cite[Corollary 5.8]{Munoz-Echaniz}) identifies the involution $\tau$ on $\pi_1 \cH(T^n)$; under (\ref{eq: piece 2.1}), $\tau$ sends each $t_i$ to $t_i^{-1}$ and is extended multiplicatively (note that $\pi_1 \cH(M)$ is not in general equipped with a product, but it is convenient in this case to identify it with the underlying abelian group of a ring).
    \begin{LL}[{\cite[Theorem 3.1]{Hat78}}]\label{lem: stab P}
        The stabilisation map $\pi_0 P(T^n) \to \pi_0 \cP(T^n)$ is an isomorphism.
    \end{LL} 
    \begin{LL}
        Under the isomorphism $\pi_0 P(T^n) \cong \pi_0 \cP(T^n)$, the kernel of $h: \pi_0 P(T^n) \to \pi_1 \widetilde{\Diff}(T^n) / \Diff(T^n)$ is exactly the subgroup $I$ generated by elements of the form $x+\tau(x)$.

        In particular under (\ref{eq: piece 2.1}),
        \begin{equation}\label{eq: piece 2}
            \pi_1 \widetilde{\Diff}(T^n) / \Diff(T^n) = (\bZ/2)[t_1, t_1^{-1}, \ldots, t_n, t_n^{-1}]/((\bZ/2)[1] + I) \cong (\bZ/2)^\infty
        \end{equation}
    \end{LL}
    \begin{proof}
        This is proved in \cite[Section 2]{Hat78} using the Hatcher spectral sequence (cf. also \cite[Section 25.3]{Kupers:book}).
    \end{proof}

\subsubsection*{Reassembling the pieces}
    The map $\Diff(T^n) \to G(T^n)$ admits a homotopy section, so the long exact sequence on homotopy groups of (\ref{eq: fib 1}) splits, and so we obtain an isomorphism:
    \begin{equation}\label{eq: reassemble 1}
        \pi_0 \Diff(T^n) \cong GL_n(\bZ) \ltimes \pi_1 G(T^n) / \Diff(T^n)
    \end{equation}
    \begin{LL}[{\cite[Section 4(4)]{Hat78}}]
        The long exact sequence on homotopy groups of (\ref{eq: fib 2}) canonically splits on $\pi_1$ in the case $M=T^n$, and so we obtain an isomorphism:
        \begin{equation}\label{eq: reassemble 2}
            \pi_1 \cS(T^n)\cong \pi_1 \widetilde{\Diff}(T^n)/\Diff(T^n)\oplus \pi_1 G(T^n)/\widetilde{\Diff}(T^n)
        \end{equation}
    \end{LL}
    \begin{proof}
        Using the fact that $G(T^n)/\widetilde{\Top}(T^n)$ is contractible \cite[Theorem V.C.2(i)]{Kirby-Siebenmann}, we may combine (\ref{eq: fib 2}) and (\ref{eq: fib 2 TOP}) to obtain a commutative diagram with exact rows:

        \begin{equation}
            \begin{tikzcd}
                \ldots
                \arrow[r]
                &
                \pi_1 \widetilde{\Diff}(T^n)/\Diff(T^n)
                \arrow[d, "\cong"]
                \arrow[r]
                &
                \pi_1 G(T^n) / \Diff(T^n)
                \arrow[r]
                \arrow[d, "q"]
                &
                \pi_1 G(T^n)/\widetilde{\Diff}(T^n)
                \arrow[d]
                \arrow[r]
                &
                \ldots
                \\
                0
                \arrow[r]
                &
                \pi_1 \widetilde{\Top}(T^n)/\Top(T^n)
                \arrow[r, "\cong"]
                &
                \pi_1 G(T^n)/\Top(T^n)
                \arrow[r]
                &
                0
                &
            \end{tikzcd}
        \end{equation}
        The left vertical arrow is an isomorphism: this follows from \cite[Theorem 3.1]{Hat78}, or alternatively can be deduced from \cite[Theorem A]{WW:I}, using the fact that $\cH(T^n) \to \cH^{\Top}(T^n)$ is 1-connected. Thus the map labelled $q$ induces a splitting $g: \pi_1 G(T^n) / \widetilde{\Diff}(T^n) \to \pi_1 G(T^n) / \Diff(T^n)$.
    \end{proof}
    \begin{RR}\label{rmk: qg0}
        Note that the splitting $g$ is compatible with the splitting of the lower exact sequence $0 \to \pi_1 G(T^n)/\Top(T^n)$; in this case, this means that the composition $q \circ g$ is 0.
    \end{RR}

    Note that Theorem \ref{thm: hat} then follows from (\ref{eq: reassemble 1}), (\ref{eq: reassemble 2}), (\ref{eq: piece 1}) and (\ref{eq: piece 2}).

\subsubsection*{1-parameter Whitehead torsion}
    
    \begin{LL}\label{lem: w factor 1}
        Under the isomorphism (\ref{eq: reassemble 2}), $w$ is injective on the first factor.
    \end{LL}
    \begin{proof}
        From the description of $\pi_0 P(T^n) \cong \pi_0 \cP(T^n)$ in (\ref{eq: piece 2.1}) and the involution on it, we see from Proposition \ref{lem: w on P} that $w \circ c \circ h: \pi_0 P(T^n) \to \pi_1 \cH(T^n)$ has kernel $I = \ker(h)$. 

        Since $h$ is surjective, it follows that $w \circ c: \pi_1 \widetilde{\Diff}(T^n)/\Diff(T^n) \to \pi_1\cH(T^n)$ is injective.
    \end{proof}
        
    \begin{LL}\label{lem: w factor 2}
        Under the isomorphism (\ref{eq: reassemble 2}), $w$ vanishes on the second factor.
    \end{LL}
    \begin{proof}
        We have the following commutative diagram:
        \begin{equation}
            \begin{tikzcd}
                \pi_1 G(T^n) / \Diff(T^n)
                \arrow[d, "w"]
                \arrow[r, "q"]
                &
                \pi_1 G(T^n)/\Top(T^n)
                \arrow[d, "w^{\Top}"]
                \\
                \pi_1 \cH(T^n)
                \arrow[r]
                &
                \pi_1 \cH^{\Top}(T^n)
            \end{tikzcd}
        \end{equation}
        The bottom horizontal arrow is an isomorphism by Lemma \ref{lem: Diff Top iso}. Since $w^{\Top} \circ q \circ g$ vanishes (cf. Remark \ref{rmk: qg0}), this implies $w \circ g$ also vanishes, as desired.
    \end{proof}

    \begin{proof}[Proof of Corollary \ref{cor: mcg tori}]
        Let $\phi$ be as in the statement of Corollary \ref{cor: mcg tori}. The projection from $\pi_0 \Diff(T^n)$ to $GL_n(\bZ)$ is given by the action on $H_1(T^n)$, so since $\phi$ is homotopic to the identity, $[\phi]$ must project to 0 in the first factor.

        Since $\phi$ is homotopic to the identity, $[\phi]$ lifts to a class $\widetilde \phi \in \pi_1 \cS(T^n)$. The set of such lifts is exactly the image of the map $\pi_1 G(T^n) \to \pi_1 \cS(T^n)$; since $\pi_1 \Diff(T^n) \to \pi_1 G(T^n)$ is surjective, this image vanishes, meaning the lift is unique.
        
        By Corollary \ref{cor: main}, $w(\widetilde \phi)$ vanishes. By Lemmas \ref{lem: w factor 1} and \ref{lem: w factor 2}, under (\ref{eq: reassemble 2}) the first factor of $\widetilde \phi$ must vanish.
    \end{proof}

\appendix

\section{Smoothness}
    In various places throughout the paper, we glue manifolds together along subspaces of their boundary. The na\"ive gluing is automatically a topological manifold, but we require smoothness. In fact, the gluing is smooth in its interior; in nice cases (such as when the boundary is smooth but with corners), the boundary can be smoothed using \cite[Appendix 6]{Wa82}. However, in some cases (such as gluing two discs $D^2$ along an interval embedded in each of their boundaries), this does not apply directly.
    \begin{LL}[Gluing]\label{lem: smooth structure on gluing}
        Let $B$ be a compact manifold, $\bM = \{M_b\}_{b \in B}$, $\bN = \{N_b\}_{b \in B}$ and $\bC = \{C_b\}_{b \in B}$ smooth fibre bundles over $B$ with fibres compact manifolds with corners. Let $i: \bC \to \partial\bM = \{\partial M_b\}_b$ and $j: \bC \to \partial\bN$ be smooth fibrewise codimention 0 embeddings. 
        
        We define
        \begin{equation}
            P_b := M_b \cup_{C_b} N_b
        \end{equation}
        where we glue using $i$ and $j$, and set $\bP = \{P_b\}_{b \in B}$.
        
        Then $\bP$ admits the structure of a smooth fibre bundle over $B$, such that the two natural inclusions $\bM, \bN \to \bP$ are smooth embeddings away from the images of $\partial \bC$.
    \end{LL}
    \begin{proof}
        We first consider the case where $B$ is a point. We start in a local model: we choose a homeomorphism $\phi: \bR_{\geq 0} \bR \times \bR \to (\bR_{\geq 0})^2$ which is a diffeomorphism away from 0.

        Now choose collar neighbourhoods $\hat{i}: C \times [0,\eps) \to M$ and $\hat{j}: C \times [0, \eps) \to N$ of $i$ and $j$, as well as $\hat{l}: \partial C \times [0, \eps) \to C$ of the inclusion of $\partial C$ (here $\eps > 0$ is some small number).

        Choose $i': \partial C \times (-\eps, \eps) \times [0,\eps) \to M$ a collar neighbourhood of $\partial C$ in $M$, restricting to $\hat i$ on $\partial C \times [0,\eps)^2$, as well as $j': \partial C \times (-\eps, \eps) \times [0,\eps) \to N$ similarly.

        We define embeddings (of topological manifolds):
        \begin{itemize}
            \item $\alpha: C \times (-\eps, \eps) \to P$ given by $\hat i \circ (\Id_C \times -\Id_{[0,\eps)})$ and $\hat j$ on $C \times (-\eps, 0]$ and $C \times [0, \eps)$ respectively.
            \item $\beta: M \setminus \hat i (C \times [0, \eps/2]) \to M \subseteq P$ and $\beta': N \setminus \hat j(C \times [0, \eps/2]) \to N \subseteq P$ given by the natural inclusions.
            \item $\gamma: \partial C \times (-\eps, \eps) \times [0,\eps)$, which we define as follows.
            \begin{itemize}
                \item On the subspace $\partial C \times (-\eps, 0] \times [0, \eps)$, $\gamma$ is the composition:
                \begin{equation}
                    \partial C \times (-\eps, 0] \times [0, \eps) 
                    \xrightarrow{\Id \times (-\Id) \times \Id}
                    \partial C \times [0, \eps)^2 
                    \xrightarrow{\Id \times \phi} 
                    \partial C \times (-\eps, \eps) \times [0, \eps)
                    \xrightarrow{\hat i} M \subseteq P
                \end{equation}
                \item On the subspace $\partial C \times [0, \eps) \times [0, \eps)$, $\gamma$ is the composition:
                \begin{equation}
                    \partial C \times [0, \eps)^2 
                    \xrightarrow{\Id \times \phi} 
                    \partial C \times (-\eps, \eps) \times [0, \eps)
                    \xrightarrow{\hat j} N \subseteq P
                \end{equation}
            \end{itemize}
        \end{itemize}

        These together define an open cover on $P$ by smooth manifolds. A direct check shows that the transition functions are smooth, and hence defines a smooth structure on $P$. 

        The case where $B$ is not a point follows from the same argument, using the fact that collar and tubular neighbourhoods exist in families.
    \end{proof}
    \begin{LL}[Isotopy extension]\label{lem: isotopy diffeo}
        Assume we are in the same situation as Lemma \ref{lem: smooth structure on gluing}. Let $k: \bC \to \partial \bN$ be another smooth fibrewise codimention 0 embedding, and let $\bQ = \{Q_b = M_b \cup_{C_b} N_b\}_b$ be given by gluing along $i$ and $k$. 

        Suppose $j$ and $k$ are isotopic through smooth fibrewise codimention 0 embeddings. Then there is a diffeomorphism of fibre bundles $f: \bP \to \bQ$.

        Suppose furthermore that $\bE \subseteq \partial \bM$ and $\bF \subseteq \partial \bN$ are smooth subbundles, disjoint from the images of $i$, $j$, $k$ and the isotopy above. Then the diffeomorphism $f$ can be chosen to restrict to the identity over $\bE$ and $\bF$.
    \end{LL}
    \begin{proof}
        Let $\{r_t\}_{t \in [0,1]}$ be the isotopy between $j$ and $k$. Let $B' := B \times I_t$, $p: B' \to B$ the projection to the first factor and let $\bM'$, $\bN'$ and $\bC'$ be the pullbacks of the bundles $\bM$, $\bN$ and $\bC$ respectively, under $p$.

        Let $R_{(b,t)}$ be given by $M_b \cup_{C_b} N_b$, where we glue along $i_b$ and $r_{b,t}$. By Lemma \ref{lem: smooth structure on gluing}, this defines a fibre bundle $\bR$ over $B'$. By construction, this is a concordance (of bundles over $B$) between $\bP$ and $\bQ$, and so they are isomorphic fibre bundles.

        The final statement follows from the construction.
    \end{proof}

    \bibliography{Refs.bib}{} \bibliographystyle{amsalpha}
    \Addresses
\end{document}